\documentclass[12pt]{amsart}
\usepackage[sorted-cites]{amsrefs}

\usepackage{graphicx}
\newtheorem{thm}{Theorem}[section]
\newtheorem{cor}[thm]{Corollary}
\newtheorem{lem}[thm]{Lemma} 
\newtheorem{prop}[thm]{Proposition}
\theoremstyle{definition}

\theoremstyle{remark}

\numberwithin{equation}{section}
\newcommand{\norm}[1]{\left\Vert#1\right\Vert}
\newcommand{\abs}[1]{\left\vert#1\right\vert}

\newcommand{\To}{\longrightarrow}

\newcommand{\PNK}{\mathbb{P}^{N_{k}}}

\newcommand{\inner}[1]{\langle #1 \rangle}
\newcommand{\R}{\mathbb R}
\newcommand{\C}{\mathbb C}

\newcommand{\ddbar}{\sqrt{-1}\,\bar{\partial}\partial}

\newcommand{\Lie}[1]{\mathfrak{#1}}

\newcommand{\Hilb}{\textrm{Hilb}_{k}}
\newcommand{\FS}{\textrm{FS}_{k}}

\begin{document}

\title[]{Relative Chow stability and extremal metrics}%
\author{REZA SEYYEDALI}%
\address{Department of Mathematics, University of Georgia, Athens, GA 30602}%
\email{rseyyeda@uga.edu}%

\thanks{}%

\date{May 19, 2017}
\begin{abstract}

We prove that the existence of extremal metrics implies asymptotically relative Chow stability. An application of this is the uniqueness, up to automorphisms, of extremal metrics in any polarization. \end{abstract}
\maketitle
\pagenumbering{arabic}

\newtheorem*{Mthm}{Main Theorem}
\newtheorem{Thm}{Theorem}[section]
\newtheorem{Prop}[Thm]{Proposition}
\newtheorem{Lem}[Thm]{Lemma}
\newtheorem{Cor}[Thm]{Corollary}
\newtheorem{Def}[Thm]{Definition}
\newtheorem{Guess}[Thm]{Conjecture}
\newtheorem{Ex}[Thm]{Example}
\newtheorem{Rmk}{Remark}
\newtheorem{Not}{Notation}
\def\thesection{\arabic{section}}
\renewcommand{\theThm} {\thesection.\arabic{Thm}}

\section{Introduction}

On a compact K\"ahler manifold $M$, extremal metrics are introduced by Calabi as canonical representations in K\"ahler classes (\cite{Ca}). Extremal metrics are critical points of Calabi functional 
$$\textrm{Cal}(\omega)=\int_{M} S(\omega)^2 \omega^n$$ restricted to a given K\"ahler class, where $S(\omega) $ is the scalar curvature of $\omega$. Extremal metrics are generalization of constant scalar curvature K\"ahler (cscK) metics . 
There is a deep relationship between the existence of canonical metrics on polarized manifolds and the concept of stability. It was conjectured by Yau (\cite{Y3}), Tian (\cite{T}) and Donaldson (\cite{D3}) that the existence of cscK metrics (and more generally extremal metrics) in a polarization is equivalent to the stability of the polarized manifold. The link between the existence and stability is provided by projective embeddings.

Let $(M,L)$ be a polarized manifold. For any $k \gg 0$ using sections of $H^{0}(M,L^k),$ there exist embeddings of $M$ into complex projective spaces. For any hermitian metric $h$ on $L$ such that  $\omega=\ddbar \log h $ is a K\"ahler form on $M$, one can use $L^2$-orthonormal bases of $H^{0}(M,L^k)$ to embed $M$ into complex projective spaces. For any such embedding, the pull back of the Fubini-Study metric to $M$ rescaled by a factor of $k^{-1}$ is a K\"ahler metric in the class of $2\pi c_{1}(L).$ In \cite{T0}, Tian proved that this sequence of rescaled metrics converges to $\omega $. In \cite{D1}, Donaldson proved that if $\omega$ has constant scalar curvature and $\textrm{Aut}(M,L)/\mathbb{C}^*$ is discrete, then there exists unique "balanced" embedding of $M$ into complex projective spaces using sections of $H^0(M,L^k)$ for $k\gg 0.$ These balanced embeddings are zeros of some finite dimensional moment maps and are essentially unique. Moreover by pulling back Fubini-Study metrics to $M$ using these embedding and rescaling by a factor of $k^{-1}$, we get a sequence of K\"ahler metrics in the class of $2\pi c_{1}(L)$ that
converges to the cscK metric.
An immediate consequence of Donaldson's theorem is the uniqueness of constant scalar curvature K\"ahler metrics in the class of
$2\pi c_{1}(L)$ under the discreteness assumption for $\textrm{Aut}(M,L)/\mathbb{C}^*.$

On the other hand, a result of Zhang(\cite{Zh}), Luo(\cite{L}), Paul (\cite{P}) and Phong and Sturm (\cite{PS1}) gives a geometric invariant theory (GIT) interpretation balanced embeddings. They show that the existence of a unique balanced metric on $L^k$ is equivalent to the Chow stability of $(M,L^k).$ Therefore, by Donaldson's theorem, the
existence of cscK metrics implies asymptotically Chow stability of $(M,L)$ under the discreteness assumption. Later, Mabuchi showed that under vanishing of some obstructions, one can drop the discreteness assumption (\cite{Ma1}, \cite{Ma3}). 
These obstructions appear if the action of the automorphism group of $M$ on the Chow line is non-trivial. In that case, any one parameter subgroup of automorphisms of $M$ that acts nontrivially on the Chow line destabilizes the Chow point. Therefore, the Chow point fails to be semi-stable. So, it is natural to study only those one parameter subgroups that are perpendicular, in some appropriate sense, to the group of automorphisms of $M$. In analogy to the Kempf-Ness theorem, Sz{\'e}kelyhidi introduced the notion of relative stability in \cite{Sz1}. Our main theorem is to prove that the existence of extremal K\"ahler metrics implies  asymptotically relative Chow stability in the sense of  \cite{Ma2} and \cite{Sz1}.
The main theorem of this article is the following. 
\begin{thm} \label{mainthm}
Let $(M,L)$ be a polarized manifold and $T \subset Aut_{0}(M,L)$ be a maximal torus. Suppose that there exists a $T$-invariant extremal K\"ahler metric $\omega_{\infty}$ in the class of $2\pi c_{1}(L)$. Then there exists a positive integer $r$ only depends on $(M,L)$ and a sequence of $T$-invariant relatively balanced metrics $\widetilde{\omega}_{k}$ on $(M,L^{rk})$ for $k\gg 0$ such that the sequence of rescaled metrics $\omega_{k}:=\frac{1}{rk}\widetilde{\omega}_{k}$ converges to $ \omega_{\infty}$ in $C^{\infty}-$topology. 
\end{thm}
Similar to the case of trivial automorphism group, we have the following. 

\begin{Cor}\label{cor1}
Let $(M,L)$ be a polarized manifold. If there exists an extremal K\"ahler metric $\omega_{\infty}$ in the class of
$2\pi c_{1}(L)$, then $(M,L^{rk})$ is relatively Chow stable for $k\gg 0.$ 

\end{Cor}

This problem was studied by Mabuchi in \cite{Ma1}-\cite{Ma5}. Mabuchi proved that the existence of extremal metrics implies a weaker version of relative Chow stability. However, this weaker version of relative Chow stability does not satisfy the uniqueness condition and therefore does not imply the uniqueness of extremal metrics.  A different approach to the problem is taken by Sano and Tipler (\cite{ST1}). They introduced the notion of $\sigma$- balanced metrics and studied its relation to modified K-energy. It was pointed out to the author by C. Tipler that their notion of $\sigma$-balanced coincides with the notion of relatively balanced. It is a consequence of their moment map interpretation of the $\sigma$-balanced metrics. Different proofs of Theorem \ref{mainthm} are given in recent papers of Mabuchi (\cite{Mnew}) and Sano and Tipler (\cite{ST2}) independently. A  closely related result is proved by Hashimoto ( \cite{Ha}, \cite{Ha2}).

Theorem \ref{mainthm} has some interesting applications. One can prove the uniqueness of extremal metrics modulo automorphisms in any polarization using approximation by relatively balanced metrics. This was conjectured by X. X. Chen for general K\"ahler classes and was proved by Berman and Berndtsson (\cite{BB}). Another application of Theorem \ref{mainthm} is a generalization of the result of Apostolov-Huang on the splitting of extremal metrics on products (\cite{AH}).  

In order to prove Theorem \ref{mainthm}, following \cite{D1}, we construct a sequence of almost relatively balanced metrics that converges to the extremal metric $\omega_{\infty}.$ Our main tools are the asymptotic expansions for the Bergman kernel (\cite{C}, \cite{Z}) and the operator $H_{k}\circ Q_{k}$ (See \cite[Lemma 2]{F}, \eqref{eq8} and Def. \ref{def4}.)  A crucial fact is that we can construct functions $F_{l}$ such that the matrix $Q_{k}(F_{l})$ induces a holomorphic vector field on $\mathbb{P}^N$ that is almost tangent to $M \subset \mathbb{P}^N$ (See \eqref{eq,F} and Theorem \ref{thm2}.) The next step is to perturb these almost relatively balanced metrics to obtain genuine solutions. In order to do that, we use the lower bound for the derivative of the moment map restricted to complement of holomorphic vector fields. This was done in \cite{Ma3} (c.f. \cite{PS2}). 

Here is the outline of the paper: In section $2,$ we review basic definitions and properties of balanced and relatively balanced metrics. In Section $3$, we review the linearization of the problem. It is essentially to find a lower bound for the derivative of the moment map. In section $4,$ we construct approximate solutions to the equation \eqref{eq,Qequation}. Section $5$ is devoted to construct almost relatively balanced metrics. We prove Theorem \ref{mainthm}
in Section $6$. Some applications of Theorem \ref{mainthm} is discussed in Section $7$.

\thanks{\textbf{Acknowledgements:} I am sincerely grateful to Vestislav Apostolov
for introducing me to the problem of stability of extremal manifolds
and invaluable suggestions. I would also like to thank Julien Keller,
G{\'a}bor Sz{\'e}kelyhidi, and Carl Tipler for many helpful discussions and
suggestions.

\section{Balanced metrics}

\subsection{Holomorphic vector fields}
Let $(M,L)$ be a polarized manifold of complex dimension $n$. Let $\omega \in 2\pi c_{1}(L)$ be a K\"ahler metric on $M$. A holomorphic vector field on $M$ is a $(1,0)$-vector field $X$ on $M$ that can be written in any local coordinate $z_{1}, \dots ,z_{n}$ as $X =\sum_{i=0}^{n} f_{i}\frac{\partial }{\partial z_{i}}$ for some holomorphic functions $f_{1},\dots , f_{n}$. A real vector field $X_{r}$ is called holomorphic if it is a real part of a holomorphic vector field $X$, i.e. $2X_{r}=X+\bar{X}.$ Note that in this case, $X =X_{r}-\sqrt{-1}JX_{r},$ since $X$ is a $(1,0)$-vector. 
We have the following.

\begin{prop}\label{propG}(\cite[Corollary 4.6]{Ko})
Let $X$ be a holomorphic vector field on $M$. The following are equivalent.

\begin{itemize}
\item $X$ can be lifted to a holomorphic vector field on $L$.

\item The zero locus of $X$ is non-empty.

\item There exists a function $f: M \to \C$ such that $\bar{\partial} f=\iota_{X}\omega.$ Such a function $f$, if exists, is unique up to a constant. The function $f$ is called a holomorphy potential for $X$.

\end{itemize}

\end{prop}
We denote the set of holomorphic vector fields satisfying the above equivalent conditions by $\Lie{g}$. Let $\widetilde{Aut}(M)$ be the group of automorphisms of $M$ that lift to $L$. Let $G=\widetilde{Aut}_{0}(M)$ be the connected component of the identity. There is a natural identification between the Lie algebra of $G$ and $\Lie{g}$. We use $\Lie{g}$ for the the Lie algebra of $G$ as well.

\begin{Def}
A holomorphic vector field $X$ on $M$ is called Hamiltonian with respect to $\omega$ if $X$ has a  real holomorphy potential, i.e. there exists a function $H: M \to \R $ such that 
$\bar{\partial} H=\iota_{X}\omega.$
\end{Def}

\begin{prop}(\cite{Ko})
A holomorphic vector field $X$ is Hamiltonian with respect to $\omega $ if and only if there exists $H:M \to \R$ such that $d H=\iota_{X_{r}}\omega,$ where $X_{r}$ is the real part of $X.$ 

\end{prop}

\begin{proof}
Let $X_{r}$ be the real part of $X$. Then 
Let $X=X_{r}-\sqrt{-1} JX_{r}.$ Let $f=u+\sqrt{-1}v$ be a holomorphy potential for $X$. We have
$$2\iota_{X_{r}}\omega - 2\sqrt{-1}\iota_{J X_{r}}\omega=2\iota_{X}\omega=2\bar{\partial } f=(du-Jdv)+\sqrt{-1} (Jdu+dv).$$ Here $J$ is the almost complex structure. Therefore,
$$2\iota_{X_{r}}\omega=du-Jdv.$$ This implies that $\iota_{X_{1}}\omega$ is exact if and only if $v$ is constant.
\end{proof}

\begin{cor}\label{cor,Hamiltonian}
Let $X$ be a Hamiltonian vector field with respect to $\omega$, $X_{r}$ be its real part  and $H \in C^{\infty} (M)$ be a Hamiltonian for $X$. Let $\phi: M \to \R$  such that  $\omega_{\phi}=\omega+\ddbar \phi$ is a K\"ahler form. Suppose that $d\phi(X_{r})=0.$ Then $X$ is Hamiltonian with respect to $\omega_{\phi}.$ Moreover, the function $H_{\phi}=H-\frac{1}{2}\inner{\nabla_{\omega}H, \nabla_{\omega}\phi}_{\omega}$ is a Hamiltonian for $X$ with respect to $\omega_{\phi}$.

\end{cor}

\begin{proof}

In this proof, all inner products and gradients are with respect to $\omega.$
Since $H$ is a Hamiltonian for $X$ with respect to $\omega,$ we have 
$ dH= 2\iota_{X_{r}}\omega$. Hence, $\nabla H=2JX_{r}$ and therefore, we have
$$\inner{\nabla H, \nabla \phi}=2d\phi (JX_{r})=2\iota_{JX_{r}}d \phi.$$
Thus, $$d \inner{\nabla H, \nabla \phi}=2d (\iota_{JX_{r}} d\phi)=2\mathcal{L}_{JX_{r}}d\phi=-2\iota_{X_{r}}\ddbar{\phi}.$$Therefore, 
$$d H_{\phi}=d H- \frac{1}{2}d \inner{\nabla H, \nabla \phi}=\iota_{X_{r}}\omega+\iota_{X_{r}}\ddbar{\phi}=\iota_{X_{r}}\omega_{\phi}.$$

\end{proof}

Extremal metrics are critical points of Calabi functional 
$$\textrm{Cal}(\omega)=\int_{M} S(\omega)^2 \omega^n$$ restricted to a given K\"ahler class, where $S(\omega) $ is the scalar curvature of $\omega.$ A straightforward calculation shows that a K\"ahler metric $\omega$ is extremal if and only if the vector field $J \nabla_{\omega} S(\omega)$ is a (real) holomorphic vector field. It is equivalent to the existence of a holomorphic vector field $X$ on $M$ satisfying  $\bar{\partial} S=\iota_{X}\omega.$ Since $S(\omega)$ is a real-valued function, the holomorphic vector field $X$ is Hamiltonian with respect to $\omega.$

\subsection{Fubini-Study metrics on complex projective spaces}
In this subsection, we fix some of notations that we use in the paper.
We start with some basic facts about complex projective spaces and Fubini-Study metrics. Tangent vectors to $\mathbb{P}^{N}$ are given by pairs
$\{(z,v)| z \in \C^{N+1}-\{0\}, v \in \C^{N+1} \}$ modulo an equivalence relation $\sim$  defined as follows:
$$(z,v)\sim(z',v') \textrm{ if } z'=\lambda z \textrm{ and }
v'-\lambda v=\mu z \textrm{ for some } \lambda \in \C^*
\textrm{ and } \mu \in \mathbb{C}.$$ For a tangent vector
$[(z,v)]$, the Fubini-Study metric is defined by
\begin{equation}\label{eq1}\norm{[(z,v)]}^2=\frac{v^*vz^*z-(z^*v)^2}{(z^*z)^2}.\end{equation} 
The Fubini-Study metric defined by \eqref{eq1} is a K\"ahler metric. We denote the coresponding K\"ahler form by $\omega_{FS}.$
There is a natural action of $U(N+1)$ on $\mathbb{P}^{N}$ that preserves $\omega_{FS}$. This action is Hamiltonian and the moment map is given by 
\begin{equation}\label{eq2}\mu(z)=\frac{zz^*}{z^*z}=\Bigg(\frac{z_{i}\bar{z}_{j}}{\abs{z}^2}\Bigg) \in \sqrt{-1} \Lie{u}(N+1).\end{equation}
For any $A \in \sqrt{-1}\Lie{u}(N+1)$, we define a holomorphic vector field $\xi_{A}$ on $\mathbb{P}^{N}$ by 
\begin{equation}\label{eq3}\xi_{A}(z)=[{z,A z}].\end{equation} 

Holomorphic vector field $\xi_{A}$ is Hamiltonian with respect to $\omega_{FS}$ and the associated Hamiltonian function is given by
 \begin{equation}\label{eq6}H(A)=tr(A\mu).\end{equation}
Moreover, the Hamiltonian $H(A)$ satisfies the following normalization condition: $$\int_{\mathbb{P}^{N}} H(A) \omega_{FS}^N=0.$$
The following is straightforward. 

\begin{lem}\label{lem2}(c.f. \cite[Lemma 11]{F})
For any $A,B \in \sqrt{-1} \Lie{u}(N+1)$  we have

$$H_{A}H_{B}+\inner{\xi_{A},\xi_{B}}_{FS}=Tr(AB\mu).$$

\end{lem}

\subsection{Balanced metrics and Chow stability}

For the polarized manifold $(M,L),$ denote the space of positive hermitian metrics on $L$ by  $\mathcal{K}_{L}$ and the space of hermitian inner products on  $H^{0}(M,L^k)$ by $\mathcal{K}_{k}$.

\begin{Def} \label{def1} 
 For any positive hermitian metric $h \in \mathcal{K}_{L},$ define the K\"ahler  form $\omega_{h}=\ddbar \log h.$
\begin{itemize}
\item Any $h \in \mathcal{K}_{L}$ defines an $L^2$ hermitian inner product $\Hilb(h)$ on $H^{0}(M,L^k)$ as follows:
\begin{align*}
\Hilb: \mathcal{K}_{L}  \to \mathcal{K}_{k}\,\,\,\,\,\,\,\,\,\,\, h \mapsto\Hilb (h)\\
\inner{ s,t } _{\Hilb(h)}= \frac{N_{k}+1}{V} \int_{M} \langle
s(x),t(x) \rangle_{h}\omega_{h}^n,
\end{align*}
where $N_{k}+1=\dim(H^{0}(M,L^k)) $ and $V=\int_{M} \omega_{h}^n.$ \\

\item Given $H \in \mathcal{K}_{k},$ we define $\FS(H)$ as the unique
metric on $L^k$ such that  $$\sum_{i=0}^{N_{k}} \norm{s_{i}}^2_{\FS(H)}=1,$$ where $s_{0},\dots ,s_{N_{k}}$ is an
orthonormal basis for $H^{0}(M,L^k)$ with respect to $H$. This defines a map $\FS:\mathcal{K}_{k} \rightarrow \mathcal{K}_{L}$. 

\end{itemize}

\end{Def}

The Aubin-Yau functional $I: \mathcal{K}_{L}\rightarrow \mathbb{R}$
is defined using the variational formula,
\begin{equation}\label{eq4}\frac{d}{dt}I(g(t))= \frac{1}{V}\int_{M} \dot{\varphi}_{t} \,\omega_{g_{t}}^n,\end{equation}
where $g_{t}=e^{\varphi_{t}} g_{0}$ is a smooth path in
$\mathcal{K}_{L}$ and $\omega_{g_{t}}=\ddbar \log g_{t}$. This functional is  unique up to a constant
which can be fixed by choosing a reference metric $g_{0}$ in
$\mathcal{K}_{L}$. By restricting the functional $-I$ to  $\FS(\mathcal{K}_{k}),$ we obtain functionals $\mathcal{L}_{k}: \mathcal{K}_{k} \to \mathbb{R}$ defined  by
\begin{equation}\label{eq5}\mathcal{L}_{k}(H)=-I \circ \FS(H).\end{equation}

\begin{Lem}\label{lem1}
 Let $H_{t}=e^{t \delta H }H$ be a path in $\mathcal{K}_{k}$, where
$\delta H$ is a hermitian matrix. We have
$$\frac{d}{dt}\Big|_{t=0}\mathcal{L}_{k}(H_{t})= \int_{M} Tr(\delta H [\langle
s_{i}, s_{j}\rangle_{FS(H)}])\omega_{FS,H}^n,$$ where $s_{0},\dots
,s_{N_{k}} $ is an orthonormal basis for $H^{0}(M,L^{k})$ with respect to
$H$ and $\omega_{FS,H} =\ddbar \log\FS(H).$

\end{Lem}

Let $h$ be a hermitian metric on $L$ and $\omega=\omega_{h}$ be the corresponding K\"ahler metric. For the rest of this section, we fix $k \gg 0$ and an orthonormal basis $s_{0},\dots s_{N}$ for $H^{0}(M,L^k)$ with respect to $L^2(h^k, \omega)$. Here, $N+1=N_{k}+1=\dim H^{0}(M,L^k)$. Using this basis, we have an embedding $\iota: M \to \mathbb{P}^N$. We denote the pull back of the Fubini-Study K\"ahler on $\PNK$ to $M$ and the Fubini-Study hermitian metric on $\mathcal{O}_{\mathbb{P}^N}(1)$ to $L^k$ by $\omega_{FS}$ and $h_{FS}$ respectively. Note that $h_{FS}=\FS \circ \Hilb (h)$ and $\omega_{FS}= \ddbar \log h_{FS,k}.$ We can identify the space of Fubini-Study metrics $\mathcal{K}_{k}$ on $H^{0}(M,L^k)$ with $$\frac{SL(N+1,\mathbb{C})}{SU(N+1)}\cong \sqrt{-1}\Lie{su}(N+1).$$
Thus, we can consider the functional $\mathcal{L}_{k}$ as a functional on $\sqrt{-1}\Lie{su}(N+1)$.
More precisely, we define $\mathcal{F}: \sqrt{-1}\Lie{su}(N+1) \to \R$ by \begin{equation}\label{eq,F}\mathcal{F}(A)=\mathcal{L}_{k}(\exp(A)), \,\,\, A \in \sqrt{-1}\Lie{su}(N+1).\end{equation}

Using the embedding $\iota: M \To \mathbb{P}^{N}$, we have the
following exact sequence of holomorphic vector bundles over $M$
$$0\rightarrow TM \rightarrow \iota^* T\mathbb{P}^{N}
\rightarrow Q \rightarrow 0.$$

Let $\mathcal{N}\subset \iota^* T\mathbb{P}^{N} $ be the
orthogonal complement of $TM$ with respect to the Fubini-Study metric on $\mathbb{P}^{N}.$ Then as smooth vector bundles, we
have
$$\iota ^* T\mathbb{P}^{N}= TM\oplus \mathcal{N}.$$ We denote the
projections onto the first and second component by $\pi_{T}$ and
$\pi_{\mathcal{N}}$ respectively. \\

\begin{lem}\cite[Lemma 3.1]{PS1}\label{lem3}
Let $A \in \sqrt{-1}\Lie{su}(N+1).$ Define $$f_{A}(t)=\mathcal{F}(tA).$$ Then we have:

$$\dot{f}_{A}:=\frac{d}{dt} f_{A}(t)=\int_{M}tr(A\mu)\sigma_{t}^*\omega_{FS}^n,$$
$$\ddot{f}_{A}:=\frac{d^2}{dt^2} f_{A}(t)=\int_{M} \norm{\pi_{\mathcal{N}} \xi_{A}}^2_{FS}\sigma_{t}^*\omega_{FS}^n.$$
Here $\sigma_{t}(z)=\exp(tA)z$, for any $z \in \mathbb{P}^N$

\end{lem}

\begin{Def}\label{def2}
An embedding $\iota: M \to \mathbb{P}^N$ is called balanced if there exists a constant $C$ such that  
$$\int_{M} \frac{z_{i}\bar{z}_{j}	}{\abs{z}^2}=C\delta_{ij}.$$
The hermitian metric $\iota ^* h_{FS}$ on $\iota ^* \mathcal{O}_{\mathbb{P}^N}(1)$ and the K\"aher form $\iota ^* \omega_{FS}$ on $M$ are called balanced as well. 

\end{Def}

Note that Lemma \ref{lem1} implies that balanced metrics on $L^k$ are exactly critical points of the functional $\mathcal{L}_{k}$. The existence of balanced metrics is closely related to Chow stability. Next, we define Chow stability.
\begin{Def}\label{def:Chowstable}

Let $M \subset  \mathbb{P}^N$ be a $n$ dimensional projective sub-variety of degree $d.$ Let $$\mathcal{Z}=\{  P \in Gr(N-n-1,  \mathbb{P}^N) | P \bigcap M \neq \emptyset\}.$$ Then $\mathcal{Z}$ is a hypersurface of degree $d$ in $Gr(N-n-1,  \mathbb{P}^N)$ and therefore there exists $f_{M} \in H^0(Gr, \mathcal{O}(d))$ such that $\mathcal{Z}=\{f_{M}=0\}. $ The point $\textrm{Chow}(M)=[f_{M}] \in \mathbb{P}(H^0(Gr, \mathcal{O}(d)))$ is called the Chow point of $M$. We say $M \subset \mathbb{P}^N$ (or equivalently $(M,\mathcal{O}_{ \mathbb{P}^N}(1)|_{M})$) is Chow stable if $[f_{M}]$ is stable under the action of $SL(N+1)$ on $\mathbb{P}(H^0(Gr, \mathcal{O}(d)))$.

\end{Def}

By a theorem of Zhang (\cite{Zh}), the existence of balanced metrics is equivalent to  (poly) stability of the Chow point of $\iota: M \to \mathbb{P}^N.$

\begin{thm} \label{thm1}(\cite{Zh}, \cite{L}, \cite{P}, \cite{PS1})
Let $\iota: M \to \mathbb{P}^N$ be a smooth projective sub-variety. The Chow point of $M$ is stable if and if there exists $\sigma \in SL(N+1, \C),$ unique up to $SU(N+1),$ such that $$\int_{\sigma M} \frac{z_{i}\bar{z}_{j}	}{\abs{z}^2}=C\delta_{ij}.$$ 
\end{thm}

\subsection{Relatively balanced metrics and stability}

In the case that the automorphism group of $M$ is not discrete, it stabilizes the Chow point of $(M,L).$ Therefore, if the group $\widetilde{Aut}(M)$ acts on the Chow line non-trivially, then the Chow point is strictly un-stable. So, in this case it is natural to only consider the subgroup in $SL(N+1, \C) $ that is "perpendicular to the image of $\widetilde{Aut}(M)$ in $SL(N+1,\C).$ This leads to the notion of relative stability of the Chow point (\cite{Sz1}, \cite{Ma1}). 
As before, let  $G=\widetilde{Aut}_{0}(M)$ be the connected component of the identity in $\widetilde{Aut}(M).$ For the rest of this article, we fix  a maximal compact torus $T\subset G$. Let $T^{\C}$ be the complexification of $T$ in $G$. We denote the Lie algebras of $T$ and $T^{\C}$ by $\Lie{t}$ and $\Lie{t}^{\C}$ respectively. By replacing $L$ with a sufficiently high power of $L$, if necessary, we may assume that the group $G$ acts on $L$ and therefore it induces an action on $H^0(M,L^k).$ We can decompose $H^0(M,L^k)$ into eigenspaces of $T^{\C}.$ More precisely, let $\chi$ be a character of $T^{\C}$. Define, $$E(\chi)= \{ s \in H^0(M,L^k)| \, t.s=\chi(t)s, \, \textrm{for all } \, t \in T^{\C} \}.$$
Therefore, there exist mutually distinct characters $\chi_{0}, \dots ,\chi_{r}$ of $T^{\C}$ such that  \begin{equation}\label{eq61}H^0(M,L^k)=\bigoplus_{i=0}^{r} E(\chi_{i}),\end{equation} Moreover, $\prod_{i=0}^r \chi_{i}^{n_{i}}=1,$ where $n_{i}=\dim E(\chi_{i}).$

\begin{Def}(c.f. \cite[pp.154-155]{AH})
An ordered basis $\underline{s}=(s_{0},\dots ,s_{N_{k}})$ for $H^{0}(M,L^k)$ is compatible with respect to the torus $T$ if for any $0\leq i \leq r$, $\{s_{n_{0}+\dots n_{i-1}}, \dots s_{n_{0}+\dots n_{i}-1}\}$ is a basis for $E(\chi_{i})$. We denote the set of all ordered bases of $H^{0}(M,L^k)$ compatible with respect to the torus $T$ by $\mathcal{B}_{k}^T$.
\end{Def}

 Fix an ordered basis $\underline{s}=(s_{0},\dots ,s_{N_{k}}) \in \mathcal{B}_{k}^T.$ Using
 $\underline{s}$, one can identify $\mathbb{P}(H^{0}(M,L^k)^*)$ and $GL(H^{0}(M,L^k)^*)$ with $\PNK$ and $GL(N_{k}+1, \C)$ respectively. This identification also induces a linearized action of $G$ on $\PNK$. Denote the induced representation of $G$ in $SL(N_{k}+1)$ by 
$$R_{\underline{s}}:G \to SL(N_{k}+1)$$ and the Lie algebra representation of $\Lie{g}$ in $\Lie{sl}(N_{k}+1)$  
by $$TR _{\underline{s}}: \Lie{g} \to \Lie{sl}(N_{k}+1)=\Lie{su}(N_{k}+1)\oplus \sqrt{-1}\Lie{su}(N_{k}+1).$$ We denote the the orthogonal projection of $TR _{\underline{s}}(X)$ on $\sqrt{-1}\Lie{su}(N_{k}+1)$ by $\Lie{R}_{\underline{s}}$. Therefore,  \begin{equation}\label{neweq1}\Lie{R}_{\underline{s}}: \Lie{g} \to  \sqrt{-1}\Lie{su}(N_{k}+1).\end{equation}
Define
\begin{equation}\label{eqST}S_{T}^{\C}=\{\textrm{diag}(A_{0}, \dots A_{k}) \in \prod_{i=0}^r GL(n_{i}, \C) | \prod_{i=0}^r \det(A_{i})=1\},\end{equation}
\begin{equation}\label{eqSTperp}S_{T^{\bot}}^{\C}=\{\textrm{diag}(A_{0}, \dots A_{k}) \in S_{T}^{\C}| \prod_{i=0}^r \det(A_{i})^{1+\log \abs{\chi_{i}(t)}}=1\,\, \textrm{for all} \,\, t \in T^{\C}\}.\end{equation}
The subgroup $S_{T}^{\C}$ is the centerlizer of $R_{\underline{s}}(T^{\C})$ in $SL(N_{k}+1, \C)$. Denote the Lie algebras of $S_{T}^{\C}$ and $S_{T^{\bot}}^{\C}$ by $\Lie{s}_{T}^{\C}$ and $\Lie{s}_{T^{\bot}}^{\C}$ respectively. It is useful to define $$\Lie{s}_{T}:= \sqrt{-1} \Lie{su}(N_{k}+1)\bigcap \Lie{s}_{T}^{\C}.$$Note that $S_{T}^{\C}$ and  $S_{T^{\bot}}^{\C}$ do not depend on the choice of $\underline{s} \in \mathcal{B}_{k}^T$ and only depend on the splitting \eqref{eq61}.
\begin{Def}
For the ordered basis $\underline{s}$, we denote the image of $\Lie{R}_{\underline{s}}$ in $\sqrt{-1}\Lie{su}(N_{k}+1)$ by $V_{\underline{s}},$  i.e.
\begin{equation}\label{eqVk}V_{\underline{s}}=\{A \in \sqrt{-1}\Lie{su}(N_{k}+1)| A=\Lie{R}_{\underline{s}}(X), \,\, \textrm{for some} \,\, X \in \Lie{g}\}.\end{equation}
Note that $V_{\underline{s}}$ is exactly the set of all matrices $A \in \sqrt{-1}\Lie{su}(N_{k}+1)$ such that $\xi_{A}$ is a holomorphic vector field on $\PNK$ tangent to $M$. We also define the orthogonal complements of $V_{\underline{s}}$ in $\sqrt{-1}\Lie{su}(N_{k}+1)$ with respect to the Killing form as follows: 
\begin{equation}\label{eqVkperp}V_{\underline{s}}^{\bot}=\{B \in \sqrt{-1}\Lie{su}(N_{k}+1)| tr(AB)=0, \textrm{for all} A \in V_{\underline{s}}\}.\end{equation}
\end{Def} 

it is more convenient to work in a $T$-invariant setting. 
\begin{Def}\label{defVkT}
Let $\underline{s} \in \mathcal{B}_{k}^T$. We denote the intersection of $V_{\underline{s}}$ and $\Lie{s}_{T}:= \sqrt{-1} \Lie{su}(N_{k}+1)\bigcap \Lie{s}_{T}^{\C}$ by $V_{\underline{s}}(T).$  We also denote the orthogonal complements of $V_{\underline{s}}(T)$ in $\Lie{s}_{T}$ by $V_{\underline{s}}(T)^{\bot}.$ 

\end{Def}

Next, we define relative Chow stability.
\begin{Def}\label{defrelativechow stability}(\cite{Sz1}, \cite{Ma1}.)
We say that the polarized manifold $(M, L^k)$ is relatively Chow stable with respect to the maximal torus $T$ if there exists an ordered basis $\underline{s} \in \mathcal{B}_{k}^T$ such that the Chow point of $\iota_{\underline{s}}: M \to\PNK$ is GIT (geometric invariant theory) stable under the action of the group $S_{T^{\bot}}^{\C}$. 
\end{Def}

The following Kempf-Ness type theorem is the analouge of Zhang's theorem in the relative case.
\begin{prop}\label{prop1}(\cite{Sz1}, \cite{Ma2})
The Chow point of $(M,L^k)$ is relatively stable with respect to the maximal torus $T$ if and only if there exists an ordered basis $\underline{s}$ for $H^{0}(M,L^k)$ such that $\displaystyle \int_{ \iota_{\underline{s}}(M)} \frac{z_{i}\bar{z}_{j}	}{\abs{z}^2}\omega_{FS}^n$ induces a holomorphic vector field on $\PNK$ tangent to $\iota_{\underline{s}}(M)$. \end{prop}
Aa an immediate consequence, the relative Chow stability does not depend on the choice of the maximal torus $T$.  
Proposition \ref{prop1} inspires the following definition.

\begin{Def}\label{def3}

Let $(M,L)$ be a polarized manifold. Suppose that $\iota: M \to \PNK$ is a Kodaira embedding using global sections of $L^k$. The embedding is called relatively balanced if the hermitian matrix $$\int_{ \iota(M)} \frac{z_{i}\bar{z}_{j}	}{\abs{z}^2}\omega_{FS}^n$$ induces a holomorphic vector field $X$ on $\PNK$ tangent to $M$.
The metric $\iota^* \omega_{FS}$ is called a relatively balanced metric on $(M,L^k)$. We also call the pair $(\iota^* \omega_{FS}, X)$ a relative balanced pair for $(M,L^k)$.

\end{Def}

\begin{Rmk}
The definition of relatively balanced metrics in \cite{Ma2} is stated differently. However, it is not hard to show that it is the same as our definition. 

\end{Rmk}

One can see that relatively balanced metrics on $(M,L^k)$, if exist, are essentially unique. The proof of the following can be found in \cite[Lemma 2]{AH}. It can be also concluded from uniqueness in relative stability (c.f. \cite[Thm. 3.5]{Sz1}).

\begin{prop}\label{propuniquness}
Suppose $\omega_{1}$ and $\omega_{2}$ are relatively balanced metrics on $(M,L^k)$. Then there exists $\Phi \in \widetilde{Aut}_{0}(M)$ such that $\Phi^* \omega_{1}=\omega_{2}.$

\end{prop}

\section{Eigenvalue estimate}

In this section, we obtain a lower bound for the second derivative of the functional $\mathcal{F}$ defined in \eqref{eq,F}. It is the same as derivative of the moment map $\mu_{D}$ introduced by Donaldson in \cite{D1}. In order to do this, we follow the argument of Phong and Sturm \cite{PS2} and Mabuchi \cite{Ma3}. The main result of this section is Theorem \ref{thm3}. 

Let $\omega_{0}$ be a $T$-invariant K\"ahler metric on $M$ in the class of $2\pi c_{1}(L)$ and $h_{0}$ be a positive hermitian metric on $L$ such that $\ddbar \log h_{0}=\omega_{0}$.

 Let $\underline{s}^{(k)}=(s_{0}^{(k)}, \dots s_{N_{k}}^{(k)}) \in \mathcal{B}_{k}^T$ be a sequence of ordered orthonormal bases with respect to $\Hilb(h_{0})$. Such bases give embeddings 
$\iota_{k}: M \To \PNK$. Note that by pulling back the FS metric on $\mathcal{O}_{\PNK}(1)$ to $L^k$ we obtain $h_{FS,k}:=\FS(\Hilb(h_{0}))$. By definition $h_{FS,k}=\FS(\Hilb(h_{0}))$ is the unique metric on $L^k$ such that $$\sum |s_{i}^{(k)}|_{h_{FS,k}}^2=1.$$ We denote the associated K\"ahler form on $M$ by $\omega_{FS,k}$ .Through this section, we fix the ordered bases $\underline{s}^{(k)}$ and associated embeddings $\iota_{k}: M \To \PNK$. We often denote the image of $M$ under this embeddings by $M$ itself.  
We have a sequence of moment maps $\mu_{k}: \PNK \to \sqrt{-1}\Lie{u}(N_{k}+1)$ for the action of $U(N_{k}+1)$ on $\PNK$. Note that the restriction of $\mu_{k}$ to $M$ is given by \begin{equation}\label{eq7}(\mu_{k})_{ij}=\inner{s_{i}^{(k)},s_{j}^{(k)}}_{\FS(\Hilb(h))}.\end{equation}

\begin{lem}[c.f. \cite{F0}, Lemma 15]\label{lem5}
Let $$\bar{\mu}_{k}=\int_{M}\langle s_{i}^{(k)}, s_{j}^{(k)} \rangle
_{h_{FS,k}} \omega_{FS,k}^n=D^{(k)}\delta_{ij}+M^{(k)}_{ij},$$
 where $D^{(k)}$ is a scalar  and $M^{(k)}$ is a trace-free hermitian
 matrix. Then $$D^{(k)}=\frac{V_{k}}{N_{k}} \rightarrow 1, \,\,\,  \norm{M^{(k)}}_{op}=O(k^{-1}) \,\,\,\,\,\textrm{ as} \,\,\, k\rightarrow \infty. $$
 
\end{lem}

Recall that we have the following exact sequence of vector bundles over $M$
$$0\rightarrow TM \rightarrow \iota_{k}^* T\mathbb{P}^{N_{k}}\rightarrow Q \rightarrow 0.$$

Let $\mathcal{N}\subset \iota_{k}^* T\mathbb{P}^{N_{k}} $ be the
orthogonal complement of $TM$. Then as smooth vector bundles, we have
$$\iota_{k} ^* T\mathbb{P}^{N}= TM\oplus \mathcal{N}.$$ We denote the
projections onto the first and second component by $\pi_{T}$ and
$\pi_{\mathcal{N}}$ respectively.

The notion of $R$-boundedness is introduced by Donaldson in \cite{D3}.
\begin{Def}\label{def11}

Let $R$ be a real number with $R >1$ and $a\geq 4$ be a fixed
integer and let $\underline{s}=(s_{0},...,s_{N})$ be an ordered
basis for $H^0(M, L^{k})$. We say $\underline{s}$ has $R$-bounded geometry if the
 K\"ahler form
  $\widetilde{\omega}=\iota^*_{\underline{s}}\omega_{\textrm{FS}}$ satisfies the following conditions
 \begin{itemize}

\item

$\norm{\widetilde{\omega}-\widetilde{\omega}_{0}}_{C^{a}(\widetilde{\omega}_{0})}\leq
R$, where $\widetilde{\omega}_{0}= k \omega_{0}$.

\item$\widetilde{\omega} \geq \frac{1}{R} \widetilde{\omega}_{0}.$

 \end{itemize}
Note that the first condition implies that $\widetilde{\omega} \leq (R+1) \widetilde{\omega}_{0}.$ Therefore, $\widetilde{\omega}$ is uniformly equivalent  to $\widetilde{\omega}_{0}$ independent of $k.$

\end{Def}
For the rest of this section, let $\underline{s}=(s_{0},\dots,s_{N})\in \mathcal{B}_{k}^T$  be a basis of $H^0(M, L^{k})$ with $R$-bounded geometry. Using the embedding $\iota_{\underline{s}}: M\to \mathbb{P}^{N_{k}}$, we can define Fubini-study metrics on $M$, $L^k$ and $\iota^* T\mathbb{P}^{N_{k}}.$ Therefore, we have the sub bundle $\mathcal{N} \subset \iota^*\mathbb{P}^{N_{k}}$  and corresponding projections $\pi_{T}$ and $\pi_{\mathcal{N}}$ on $TM$ and 
$\mathcal{N}$ respectively.

\begin{thm}\label{thm3}
  For any $R>1$, there are positive constants $C$ and  $\epsilon $
 such that, if the basis $\underline{s}=(s_{0},...,s_{N}) \in \mathcal{B}_{k}^T$ has $R$-bounded geometry, and if $\norm{\bar{\mu}(\underline{s})}_{\textrm{op}}\leq\epsilon
 $, then $$Ctr(A^2)\leq k^2\norm{\pi_{\mathcal{N}}\xi_{A}}^2,$$ for all $A \in V_{\underline{s}}(T)^{\bot} $ (Definition \ref{defVkT}). 
  
\end{thm}

For any $\underline{s},$ define the $L^2$-orthogonal complement of $V_{\underline{s}}(T) $ by
$$W_{\underline{s}}(T)=\{B \in \Lie{s}_{T}| \int_{M}  \inner{\xi_{A},\pi_{T}\xi_{B}}_{\omega_{0}} \omega_{0}^n=0, \textrm{for all} \, A \in V_{\underline{s}} (T)\}.$$
Note that Theorem \ref{thm3} will follow from the following.
\begin{equation}\label{eq20}\norm{A}^2\leq c_{R} k\norm{\xi_{A}}^2, \end{equation}
\begin{equation}\label{eq21} c^{'}_{R}\norm{\pi_{T}\xi_{A}}^2\leq k\norm{\pi_{\mathcal{N}}\xi_{A}}^2, \,\,\, \textrm{for}\,\, A \in W_{\underline{s}}(T).\end{equation}
For a proof of \eqref{eq20}, we refer the reader to \cite[p. 703-705 ]{PS2} (Proposition \ref{prop11} below). We will prove \eqref{eq21} in Proposition \ref{prop13}. Assuming these, we give
the proof of Theorem.\ \ref{thm3}.

\begin{proof}[Proof of Theorem \ref{thm3}]
Proposition \ref{prop11} and Proposition \ref{prop13} imply that \begin{equation}\label{eq11}ctr(B^2)\leq k^2||\pi_{\mathcal{N}}\xi_{B}||^2,\,\, \textrm{ for all B} \in W_{\underline{s}}(T).\end{equation} Given $A \in V_{\underline{s}}(T)^{\perp}$, there exist $A_{1} \in V_{\underline{s}}(T)$ and $A_{2} \in W_{\underline{s}}(T)$ such that $A=A_{1}+A_{2}$ since $V_{\underline{s}}(T)\oplus W_{\underline{s}}(T)=\Lie{s}_{T}.$ By definition, we have $tr(AA_{1})=0$ and $\pi_{\mathcal{N}}\xi_{A_{1}}=0$. Hence, \begin{align*} k^2\norm{\pi_{\mathcal{N}}\xi_{A}}^2 & =k^2\norm{\pi_{\mathcal{N}}\xi_{A_{1}}+\pi_{\mathcal{N}}\xi_{A_{2}}}^2=\norm{\pi_{\mathcal{N}}\xi_{A_{2}}}^2\\&\geq C tr(A_{2}^2)=Ctr((A-A_{1})^2) \\&=Ctr(A^2)+Ctr(A_{1}^2)-2Ctr(AA_{1})=Ctr(A^2)+Ctr(A_{1}^2)\\&\geq Ctr(A^2).\end{align*}

\end{proof}

\begin{prop}(\cite[p. 703-705 ]{PS2})\label{prop11}
Under the assumptions of Theorem \ref{thm3}, there exists a positive constant $c_{R}$ such that for any $A\in \sqrt{-1}\Lie{su}(N+1)$, we have
$$ \norm{A}^2\leq c_{R} k\norm{\xi_{A}}^2,$$
 where $\norm{ . }$ in the right hand side denotes the $L^2$- norm
 with respect to the K\"ahler form $\widetilde{\omega}$ on $M$ and
 Fubini-Study metric on the fibers.

\end{prop}

\begin{Rmk}

Proposition \ref{prop11} holds even if $\textrm{Aut}(M,L)$ is not discrete.  

\end{Rmk}

\begin{prop}(\cite[p.705-708]{PS2})\label{prop12}
 For any holomorphic vector field $V$ on $\mathbb{P}^{N}$, we
 have
 $$c_{R} \abs{\pi_{\mathcal{N}} V}^2 \geq   \abs{\overline{\partial}(\pi_{\mathcal{N}}
 V)}^2.$$

\end{prop}

\begin{prop}(\cite[p. 705-708]{PS2} and \cite{Ma3})\label{prop13}
Under the assumptions of Theorem \ref{thm3}, there exists a constant $c_{R}$ such that for any $A \in W_{\underline{s}}(T)$, we have
$$ c_{R}\norm{\pi_{T}\xi_{A}}^2\leq k\norm{\pi_{\mathcal{N}}\xi_{A}}^2.$$

\end{prop}

\begin{proof}
We follow the argument of Phong and Sturm \cite{PS2} and Mabuchi \cite{Ma3}.
Let $\lambda$ be the first nonzero eigenvalue of $\Delta_{\bar{\partial}}$ on $\Gamma^T=\Gamma^T(M,TM)$ with respect to the metric $\omega_{0}$ on $M$ and $TM$, where $\Gamma^T(M,TM)$ is the space of $T$-invariant vector fields on $M$. Let $\Gamma_{H}^T$ be the sub-space of smooth $T$-invariant vector fields $W$ such that $\bar{\partial}f=\iota_{W}\omega_{0}$ for some $f:M \to \C.$ 
Therefore, an argument similar to the one given in \cite[p.p. 708-710]{PS2} implies that
$$\lambda \norm{W}^2_{L^2(\omega_{0})} \leq \norm{\overline{\partial}W}^2_{L^2(\omega_{0})},$$
for any $W \in \Gamma_{0}^T$. Here,  $$\Gamma_{0}^T=\{ W \in \Gamma_{H}^T| \int_{M}\inner{W,X}_{\omega_{0}} \omega_{0}^n=0, \forall X \in \Gamma_{H}^T \bigcap \Lie{g}=\Lie{t}^{\C}\}.$$ Hence,

\begin{align*}\lambda \norm{W}^2_{L^2(\widetilde{\omega}_{0})}&=\lambda \int
\abs{W}_{\widetilde{\omega}_{0}}^2 \widetilde{\omega}_{0}^n =
\lambda k^{n+1}\int \abs{W}_{\omega_{0}}^2 \omega_{0}^n \leq k^{n+1}\int
\norm{\overline{\partial}W}_{\omega_{0}}^2 \omega_{0}^n\\&=k\int
\norm{\overline{\partial}W}_{\widetilde{\omega}_{0}}^2
\widetilde{\omega}_{0}^n=
k\norm{\overline{\partial}W}^2_{L^2(\widetilde{\omega}_{0})}.\end{align*}
Therefore, there exists a positive constant $c_{R}$ depends on $R$ and
independent of $k$, such that for any $\widetilde{\omega}$
having $R$-bounded geometry and any $W \in  \Gamma_{0}^T,$ we have

$$c_{R}\norm{W}^2_{L^2(\widetilde{\omega})} \leq k
\norm{\overline{\partial}W}^2_{L^2(\widetilde{\omega})}.$$
 For $A \in W_{\underline{s}}(T)$, we have $ \pi_{T} \xi_{A} \in \Gamma_{0}^T$(c.f. \cite[p. 710]{PS2}). Hence, 
$$c_{R}\norm{\pi_{T}\xi_{A}}_{L^2(\widetilde{\omega})}^2\leq
k\norm{\overline{\partial}(\pi_{T}\xi_{A})}_{L^2(\widetilde{\omega})}^2.$$
Applying Proposition \ref{prop12} implies that
\begin{align*}c\norm{\pi_{T}\xi_{A}}_{L^2(\widetilde{\omega})}^2&\leq
k\norm{\overline{\partial}(\pi_{T}\xi_{A})}_{L^2(\widetilde{\omega})}^2=k\norm{\overline{\partial}(\pi_{\mathcal{N}}\xi_{A})}_{L^2(\widetilde{\omega})}^2\\&\leq c_{R}k\norm{\pi_{\mathcal{N}}\xi_{A}}_{L^2(\widetilde{\omega})}^2.\end{align*}

\end{proof}

\section{Asymptotic Expansions}
 
The main goal of this section is to construct approximate solutions to the equation \eqref{eq,Qequation}. It is done in Theorem \ref{thm2}. In order to construct such approximate solutions, we first construct approximate solutions for the equation $H_{k}(Q_{k}(F))=H_{k}(A)$ (c.f. \eqref{eq8} and Definition \ref{def4}). Then we use the fact that 
the maps $$H_{k}: \sqrt{-1}\Lie{u}(N_{k}+1) \to C^{\infty}(M,\mathbb{R})$$ are injective. We are also need a uniform lower bound for $||H_{k}||_{op}$. Unfortunately there is no positive $c,M$ such that 
$$ck^{-M}tr(A^2) \leq ||H_{k}(A)||^2_{L^2}.$$ A simple example is $\mathbb{CP}^1$ (c.f. \cite{KMS}). However, if we restrict the domain of $H_{k}$ to $Q_{k}(W)$ for a finite dimensional subspace $W \subset C^{\infty}(M)$, then we can obtain a uniform lower bound for $||H_{k}||_{op}$ (Proposition \ref{prop6}).
\subsection{A lower bound on Hamiltonians}
Let $ h$ be a positive hermitian metric on $L$ and $\omega=\ddbar \log(h)$ be the corresponding K\"ahler form. Let $\underline{s}^{(k)}=(s_{0}^{(k)}, \dots ,s_{N_{k}}^{(k)})$ be a sequence of ordered orthonormal basis  for $H^0(L^k)$ with respect to $\Hilb(h)$. Such bases give embeddings 
$\iota_{k}: M \To \PNK$. Therefore, we have sequences $h_{FS,k}:=\FS(\Hilb(h))$ and $\omega_{FS,k}$ of hermitian metrics and K\"ahler forms respectively. We also have a sequence of maps $H_{k}: \sqrt{-1}\Lie{u}(N_{k}+1) \to C^{\infty}(M)$ given by 
\begin{equation}\label{eq8}H_{k}(A)= tr(\mu_{k}A)=\sum_{i,j} A_{ij}\inner{s_{j}^{(k)}, s_{i}^{(k)}}_{\FS(\Hilb(h))}.\end{equation}

\begin{Def}\label{def4}

Let  $f \in C^{\infty}(M, \R)$ and $\underline{s}^{(k)}=(s_{0}^{(k)}, \dots ,s_{N_{k}}^{(k)})$ be an orthonormal basis for $H^0(M,L^k)$ with respect to $\Hilb(h).$ Define $Q_{k}(f) \in \sqrt{-1}\Lie{u}(N_{k}+1)$ by $$\big(Q_{k}(f) \big)_{ij}=\int_{M} f\inner{s_{i}^{(k)}, s_{j}^{(k)}}_{h ^k}\omega^n.$$
\end{Def}

Note that our definition is slightly different from the one in \cite{F}. Similar calculation as in 
\cite{F} concludes the following proposition. 
\begin{prop}[ \cite{F}, Lemma 15] \label{prop4}

We have an asymptotic expansion 

$$H_{k}(Q_{k}(f)) \sim f+q_{1}(f)k^{-1}+q_{2}(f)k^{-2}+\cdots.$$ The expansion is uniform if $f$ varies in a compact set in $C^{\infty}-$topology. It is also uniform with respect to the K\"ahler metric $\omega.$
 Moreover,
$q_{1}(f)=-2\Delta f.$

\end{prop}

\begin{proof}
The kernel $K_{f,k}(x)$ is defined in  \cite{MM} as follows (See also \cite{F}):
$$K_{f,k}(x)=\sum_{i,j}\int_{M} f(y)\inner{s_{i}(y), s_{j}(y)}_{h ^k}\inner{s_{j}(x), s_{i}(x)}_{h ^k}dvol_{\omega}(y).$$Ma and Marinescu proved that there is an asymptotic expansion \begin{equation}\label{eq,K}K_{f,k}=k^n f+k^{n-1} \tilde{q}_{f,1}+\cdots, \end{equation} where $\tilde{q}_{f,i}$ are smooth functions on $M.$ In particular $\tilde{q}_{f,1}=S(\omega)f-2\Delta f.$ Moreover, the expansion is uniform if $f$ and $\omega$ vary in compact sets.  
Applying Catlin-Tian-Yau-Zelditch asymptotic expansion for the Bergman kernel $\rho_{k}(h)$ and Ma-Marniscu expansion \eqref{eq,K}, we have \begin{align*}H_{k}(Q_{k}(f))&=\rho_{k}(h)^{-1}K_{f,k}=(1-S(\omega)k^{-1}+\cdots)(f+\tilde{q}_{f,1}k^{-1}+O(k^{-2}))\\&=f+\Big(\tilde{q}_{f,1}-S(\omega) f\Big)k^{-1}+O(k^{-2})=f-2f\Delta f k^{-1}+O(k^{-2}).\end{align*}  
The first equality holds since $\FS(\Hilb(h))=\rho_{k}(h)^{-1}h^k$ and $$H_{k}(Q_{k}(f))=\sum_{i,j}\big(\int_{M} f\inner{s_{i}, s_{j}}_{h ^k}\omega^n\big)\inner{s_{j}, s_{i}}_{\FS(\Hilb(h))}.$$
\end{proof}

\begin{prop}\label{cor3}

Let $W \subset C^{\infty}(M, \mathbb{R})$ be a finite dimensional subspace. There exists a constant $c$ depends only on $W$ such that  $$\abs{tr(Q_{k}(f)^2)-k^{n}\norm{f}^2_{L^2} }\leq ck^{n-1}\norm{f}^2_{L^2},$$  for any $f \in W.$

\end{prop}

\begin{proof}

By definition, we have \begin{align*}&tr(Q_{k}(f)^2)\\&= \sum_{i,j}\int_{M} \int_{M} f(x)f(y)\inner{s_{i}(x), s_{j}(x)}_{h ^k}\inner{s_{j}(y), s_{i}(y)}_{h ^k}dvol_{\omega}(x)dvol_{\omega}(y)\\&=\int_{M} f(x) K_{f,k}(x) dvol_{\omega}(x)=k^n \int_{M} f^2 \omega^n+O(k^{n-1}).\end{align*}
Therefore, uniformity of the asymptotic expansion \eqref{eq,K} concludes the proposition.  
\end{proof}

\begin{Def}

Let $X$ be a holomorphic vector field on $M$. Suppose that $X$ is Hamiltonian with respect to $\omega.$ We define the normalized  Hamiltonian $H_{X}$ by $$\bar{\partial} H_{X}=\iota_{X} \omega,\,\,\,\, \int_{M}H_{X}\omega^n=0.$$

\end{Def}

Note that if $A=R_{k}(X) \in V_{k}$ is the corresponding hermitian matrix in $\sqrt{-1}\Lie{su}(N_{k}+1),$ then $H_{k}(A)$ is a Hamiltonian for $X$ with respect to $\omega_{FS,k}$. However, it does not necessarily satisfy any normalization on $M$. On the other hand the sequence of metrics $k^{-1}\omega_{FS,k}$ converges to $\omega$ (\cite{T0}). Therefore, we expect a relationship between $H_{X}$ and $H_{k}(A)$. We have the following asymptotic expansion.

\begin{prop}\label{prop5}

Let $X$ be a Hamiltonian holomorphic vector field on $M$ and let $$A=\Lie{R}_{k}(X) \in V_{k}$$ be the corresponding hermitian matrix in $\sqrt{-1}\Lie{su}(N_{k}+1)$ (c.f. \eqref{neweq1}). 
Define  $c_{A}(k)=\frac{1}{V}\int_{M} H_{k}(A) \omega^n.$
Then  we have the following asymptotic expansion.
$$H_{k}(A-c_{A}(k)I) \sim k \big(H_{X}+k^{-2}h_{2}(X)+\cdots \big)$$ which holds in $C^{\infty}$ and is uniform if $X$varies in a compact set. Moreovere, $\int_{M}h_{i}\omega^n=0$ and there exist a constant $c$ independent of $k$ such that 
\begin{equation}\label{eqH}\abs{c_{A}(k)} \leq ck^{-\frac{n+2}{2}}tr(A^2)^{\frac{1}{2}}.\end{equation}

\end{prop}

\begin{proof}

Without loss of generality, we may assume that $X$ belongs to the unit ball $\{\xi \in \Lie{g}| \int_{M} \norm{\xi}^2_{\omega} \omega^n \leq 1\}$. We know that \begin{align*}\omega_{FS,k}&=k\omega+\ddbar \log \rho_{k}(\omega)\\&=k\omega+\ddbar \log(1+a_{1}k^{-1}+a_{2}k^{-2}+\dots)\\&=
k\omega+\bar{\partial}(k^{-1}\partial a_{1}+\dots)\\&=k\omega+\bar{\partial}(k^{-1}\theta_{1}+\dots),\end{align*} where $\theta_{1}=\partial a_{1} , \dots.$ By definition of $H_{k}(A)$ (c.f. \eqref{eq8}), we have \begin{align*} \bar{\partial} H_{k}(A)&=\iota_{X}\omega_{FS,k}= k \iota_{X}\omega +k^{-1}\iota_{X} \ddbar  a_{1}+\dots\\&=\bar{\partial} \Big(kH_{X}+k^{-1}\theta_{1}(X)+\dots\Big)\\=k \bar{\partial} \Big(H_{X}+k^{-2}b_{2}+\dots).\end{align*} 
Define $h_{i}=b_{i}- \frac{1}{V}\int_{M} b_{i}\omega^n.$
Hence, we have $$\bar{\partial} H_{k}(A-c_{A}(k)I )=\bar{\partial} H_{k}(A)=k \bar{\partial} \Big(H_{X}+k^{-2}h_{2}+\dots).$$ On the other hand Lemma \ref{lem5}, Proposition \ref{prop11} and \eqref{eq8} imply that  \begin{align*}\abs{\int_{M} H_{k}(A)\omega^n}&=\abs{\int_{M} tr(\mu_{k}A)\omega^n}\\&=\abs{tr\Big(\big(\int_{M}\mu_{k}\omega^n -I\big)A\Big)}
\\&\leq ck^{-n}\abs{tr\Big(\big(\int_{M}\mu_{k}\omega_{FS,k}^n-I\big)A\Big)} \\&\leq ck^{-\frac{n+2}{2}}tr(A^2)^{\frac{1}{2}}\\&\leq c(\int_{M}\norm{X}^2\omega^n)^{\frac{1}{2}}.\end{align*} We have used the facts that $\norm {\omega-k^{-1}\omega_{FS,k}}=O(k^{-2})$ and $tr(A)=0.$ Note that Proposition \ref{prop11} implies that $$tr(A^2) \leq ck\int_{M} \norm{X}_{\omega_{FS,k}}^2(k\omega)^n\leq ck^{n+2}\int_{M}\norm{X}_{\omega}^2\omega^n.$$

\end{proof}

\begin{Rmk}\label{newrmk1}

It is straightforward to show that there exist  real numbers $c_{0}(A), c_{1}(A), \dots$ independent of $k$ such that the following asylumptotic expansion holds.
\begin{equation}\label{neweq2}c_{A}(k)\sim c_{0}(A)+k^{-1}c_{1}(A)+\cdots.\end{equation}

\end{Rmk}

In the next Proposition, we prove a uniform lower bound for $||H_{k}||_{op}$ restricted to "uniformly finite dimensional subspaces". 
\begin{prop}\label{prop6}

Let $W \subset C^{\infty}(M, \mathbb{R})$ be a finite dimensional subspace. There exists a constant $c$ only depends on $W$ such that for  any Hamiltonian holomorphic vector field $X$ on $M,$ $f \in W$ and $k \gg 0,$ we have $$ck^{-n}tr((Q_{k}(f)-k^{-1}A)^2) \leq ||H_{k}(Q_{k}(f)-k^{-1}A)||^2_{L^2},$$ where $A=\Lie{R}_{k}(X) \in V_{\underline{s}^{(k)}}$ is the corresponding hermitian matrix in $\sqrt{-1}\Lie{su}(N_{k}+1).$.

\end{prop}

The first step is to prove the following Lemma.

\begin{lem}\label{lem6}
There exists a positive constant $c$ such that such that for any Hamiltonian holomorphic vector field $X$ on $M,$ $\lambda \in \R$ and $k \gg 0,$ we have $$ck^{-n}tr((Q_{k}(\lambda H_{X})-k^{-1}A)^2) \leq ||H_{k}(Q_{k}(\lambda H_{X})-k^{-1}A)||^2_{L^2},$$ where $A=\Lie{R}_{k}(X) \in V_{\underline{s}^{(k)}}$ and $H_{X}$ is the normalized Hamiltonian of $X$.

\end{lem}

\begin{proof}


Using equivariant Riemann-Roch, along the line of calculation in \cite[Section 2.2]{D3}, one can prove that $$tr(A^2)= \int_{M}H_{X}^2 \omega^n k^{n+2}+\int_{M}\Big(H_{X}^2S(\omega) + \norm{X}^2 \Big) \omega^n k^{n+1}+O(k^n).$$
Let $f \in C^{\infty}(M, \mathbb{R})$ and $X$ be a holomorphic vector field such that $\int_{M}H_{X}^2 \omega^n=1.$
Proposition \ref{prop4} and \ref{prop5} imply that we have the following asymptotic expansion.  

$$ H_{k}(Q_{k}(f)-k^{-1}A)= f-H_{X}+(q_{1}(f)-c_{}(A))k^{-1}+\dots.$$ Here, $c_{0}(A)$ is defined by \eqref{neweq2}.
On the other hand, by definition of $Q_{k}(f)$, we have
 \begin{align*}k^{-n-1}tr(Q_{k}(f)A)&=k^{-n-1}tr(A\int_{M} f \inner{s_{i}, s_{j}}_{h ^k}\omega^n)\\&=k^{-n-1}tr(A\int_{M} f \rho_{k}(h)\inner{s_{i}, s_{j}}_{h_{FS}}\omega^n)
\\&=k^{-n-1}\int_{M} f \rho_{k}(h)tr(\mu_{k}A)\omega^n=k^{-n-1}\inner{f, \rho_{k}(h)H_{k}(A)}\\&=\inner{f,H_{X}}+k^{-1}\inner{f, S(\omega)H_{X}}+O(k^{-2}).\end{align*} Therefore,
\begin{align*}k^{-n} &tr((Q_{k}(f)-k^{-1}A)^2)=k^{-n}tr((Q_{k}(f)^2)+k^{-n-2}tr(A^2)-2k^{-n-1}tr(Q_{k}(f)A)
\\&=\norm{f-H_{X}}^2+\Big(  \norm{(f-H_{X})^2 S(\omega)}+\norm{X}^2-2\int_{M} f\Delta f\omega^n \Big)k^{-1}+O(k^{-2}).\end{align*}

Define $$\phi(t)= k^{-n} tr((Q_{k}((1+t)H_{X})-k^{-1}A)^2),$$
 $$\psi(t)=\norm{H_{k}(Q_{k}((1+t)H_{X})-k^{-1}A)}^2_{L^2}.$$
 Note that $\phi(0)=\norm{q_{1}(H_{X})-c_{0}(A)}k^{-2}+O(k^{-3})$ and $\psi(0)=O(k^{-2}),$ since $$2\int_{M} H_{X} \Delta H_{X} \omega^n=\int_{M} \norm{\nabla H_{X}}^2 \omega^n= \norm{X}_{L^2}^2.$$Therefore, there exist positive constants $c_{1}$ and $c_{2}$ such that 
 $$\phi(0)-c_{1}\psi(0) \geq c_{2} .$$ Here, we use the fact that there exists a constant $c$ such that $$\min_{X} \norm{q_{1}(H_{X})-c_{0}(A)}_{L^2} \geq c>0.$$
 
 Let $\Phi(t)=\phi(t)-c_{1}\psi(t).$ We have $$\Phi(t)=(1-c_{1}) t^2 \norm{H_{X}}^2+O(k^{-1})=(1-c_{1}) t^2 +O(k^{-1}).$$
 Therefore, for $k \gg 0$, we have
 
 $$\Phi(0)\geq c_{2}, \,\,\,\, \abs{\Phi^{\prime}(0)}\leq c_{3}k^{-1}, \,\,\Phi^{\prime \prime}(t) \geq \frac{1-c_{1}}{4}.$$
 Hence, \begin{align*}\Phi(r)&=\Phi(0)+r\Phi^{\prime}(0)+\int_{0}^r \int_{0}^t \Phi^{\prime \prime}(s) ds\\&\geq (c_{2}-rk^{-1}c_{3}+\frac{1-c_{1}}{4}r^2)  \geq 0.\end{align*}
 Note that the quadratic $c_{2}-rk^{-1}c_{3}+\frac{1-c_{1}}{4}r^2$ is always positive for $k \gg 0.$

\end{proof}

\begin{proof}[Proof of Proposition \ref{prop6}]

Fix a Hamiltonian holomorphic vector field $X$ on $M$ such that $\int_{M} H_{X}^2\omega^n=1.$
Define $$\widetilde{W}=\{ g \in W+\textrm{Ker}(\mathcal{D}^*\mathcal{D})|\int_{M}g H_{X} \omega^n=0\}.$$
Therefore, for any $g \in \widetilde{W},$ we have 
\begin{align*}\inner{H_{k}(Q_{k}(g)), H_{k}(Q_{k}(\lambda H_{X})-k^{-1}A) }_{L^2}&=\inner{g, (\lambda-1) H_{X} }_{L^2}+O(k^{-1})\\&=O(k^{-1}),\end{align*}
\begin{align*}k^{-n}tr \Big(Q_{k}(g)(Q_{k}(\lambda H_{X})-k^{-1}A)\Big)&=k^{-n}tr \Big(Q_{k}(g)(Q_{k}(\lambda H_{X})\Big)-k^{-n-1}tr(\Big(Q_{k}(g) A)\Big)\\&=\inner{g, \lambda H_{X}- \rho_{k}(h)H_{X}}_{L^2}+O(k^{-1})=O(k^{-1}).\end{align*} Thus, there exists a constant $c$ such that 
\begin{equation}\label{neweq3}\abs{\inner{H_{k}(Q_{k}(g)), H_{k}(Q_{k}(\lambda H_{X})-k^{-1}A) }_{L^2}} \leq ck^{-1}\norm{g}_{L^2}(\abs{\lambda}+1),\end{equation}\label{neweq3}
\begin{equation}\label{neweq4}k^{-n}\abs{tr \Big(Q_{k}(g)\big(Q_{k}(\lambda H_{X})-k^{-1}A\big)\Big)} \leq ck^{-1}\norm{g}_{L^2}(\abs{\lambda}+1).\end{equation}

For $f \in W,$ there exists $\lambda \in \R $ and  $g \in C^{\infty}(M, \mathbb{R})$ such that $f=g+\lambda H_{X},$ and $\int_{M}g H_{X} \omega^n=0.$ Therefore,

\begin{align*} \norm{H_{k}(Q_{k}(f)-k^{-1}A)}^2_{L^2}&= \norm{H_{k}(Q_{k}(g))+H_{k}(Q_{k}(\lambda H_{X})-k^{-1}A)}^2_{L^2} \\&\geq c \Big(\norm {H_{k}(Q_{k}(g)}^2_{L^2}+\norm{H_{k}(Q_{k}(\lambda H_{X})-k^{-1}A))}^2_{L^2}\Big)\\& \geq c\Big(\norm{g}^2_{L^2}+k^{-n}tr((Q_{k}(\lambda H_{X})-k^{-1}A)^2))\Big)\\&\geq c\Big(k^{-n}tr(Q_{k}(g)^2)+k^{-n}tr((Q_{k}(\lambda H_{X})-k^{-1}A)^2)\Big)\\&\geq ck^{-n} tr((Q_{k}(g+\lambda H_{X})-k^{-1}A)^2)\\&=ck^{-n} tr((Q_{k}(f)-k^{-1}A)^2)\end{align*}

The first inequality follows from \eqref{neweq3} and the last inequality follows from \eqref{neweq4}.

\end{proof}

\subsection{Constructing approximate solutions for equation \eqref{eq,Qequation}}
For a given holomorphic vector field $X$ on $M$, we would like to find $F=F_{k}$ satisfying the equation \begin{equation}\label{eq,Qequation}\xi_{Q_{k}(F)}=k^{-1}X.\end{equation} We also like to have a nice asymptotic expansion for $F$. Suppose $A=A(k) \in \sqrt{-1}\Lie{su}(N_{k}+1)$ be the corresponded matrices representing the holomorphic vector field $X$, i.e. $\Lie{R}_{k}(X)=A(k).$ Therefore \eqref{eq,Qequation} is equivalent to the equation \begin{equation}\label{eq,Qequation1}Q_{k}(F)=k^{-1}A+\lambda I.\end{equation} for some $\lambda \in \R.$
\begin{Def}\label{defF}

For a K\"ahler metric $\omega$ and holomorphic vector field $X$, Hamiltonian with respect to $\omega,$ we define $f_{i}(X, \omega) $ recursively as follows:

$$f_{0}(X, \omega)=H_{X},$$
$$f_{1}(X,\omega)= -q_{1}(f_{0}),$$
$$\vdots$$

$$f_{l}(X,\omega)= h_{l}-\sum_{i=1}^lq_{i}(f_{l-i}),$$
$$\vdots$$

Here $q_{1}, q_{2}\dots,  h_{2}, h_{3}\dots$ are given by Proposition \ref{prop4} and Proposition \ref{prop6}.
We also define functions $F_{l}(X,\omega)$ by
 $$F_{l}(X,\omega)= H_{X}+\sum_{j=1}^l k^{-j} f_{j}(X,\omega)$$
 
\end{Def}
We prove in Theorem \ref{thm2} that $F_{l}(X,\omega)$ are indeed approximate solutions for the equation \eqref{eq,Qequation}.

\begin{Rmk}Since maps $Q_{k}: C^{\infty}(M) \to \sqrt{-1}\Lie{u}(N_{k}+1)$ have large kernels, we can not expect to obtain unique solutions for the equation \eqref{eq,Qequation}. A crucial fact is that $Q_{k}$ are asymptotically invertible" with inverse $H_{k}$ in the sense of Prop. \ref{prop4}. 
Suppose that $F$ is a solution to the equation \eqref{eq,Qequation} and has an asymptotic expansion $F=\sum_{i=0}^{\infty}k^{-i}f_{i}.$
Therefore, we have $$Q_{k}(\sum_{i=0}^{\infty}k^{-i}f_{i})=k^{-1}A(k)+\lambda I.$$ Applying the map $H_{k},$ we obtain $$H_{k}(Q_{k}(\sum_{i=0}^{\infty}k^{-i}f_{i}))=k^{-1}H_{k}(A)+\lambda.$$ Applying Prop. \ref{prop4} and \ref{prop5}, we have \begin{align*}0&=H_{k}(Q_{k}(\sum_{i=0}^{\infty}k^{-i}f_{i}))-k^{-1}H_{k}(A-c_{A}(k)I)\\&=\sum_{i=0}^{\infty}k^{-i}f_{i}+q_{1}(\sum_{i=0}^{\infty}k^{-i}f_{i})k^{-1}+ q_{2}(\sum_{i=0}^{\infty}k^{-i}f_{i})k^{-2}+\cdots\\&-(H_{X}+k^{-2}h_{2}+k^{-3}h_{3}+\cdots)\\&=(f_{0}-H_{X})+(f_{1}+q_{1}(f_{0}))k^{-1}+(f_{2}+q_{1}(f_{1})+q_{2}(f_{0})-h_{2})k^{-2}\\&+(f_{3}+q_{1}(f_{2})+q_{2}(f_{1})-h_{3})k^{-3}+\cdots.\end{align*}
Now we can solve for $f_{i}$ by setting all coefficients on the right hand side equal to $0$. Therefore $F=\sum_{i=0}^{\infty}k^{-i}f_{i}$ solves the equation \eqref{eq,Qequation} formally. 
That is the motivation for defining $f_{i}$ in Definition \ref{defF}.

\end{Rmk}

The series $\sum_{i=0}^{\infty}k^{-i}f_{i}$ does not converge necessarily. However, Theorem \ref{thm2} shows that the finite sums$$F_{l}(X,\omega)= H_{X}+\sum_{j=1}^l k^{-j} f_{j}(X,\omega)$$ are approximate solutions for the equation \eqref{eq,Qequation}.

\begin{thm}
\label{thm2}

Let $X \in \Lie{g}$ be Hamiltonian with respect to $\omega$ . Let $A=A(k) \in \sqrt{-1}\Lie{su}(N_{k}+1)$ be the corresponded matrices. Then there exist constants $c_{k} \in \R$ such that for any positive integer $l$, we have  $$tr\Big(\Big(Q_{k}(F_{l}(X,\omega))-k^{-1}A(k)-c_{k}I\Big)^2\Big)=O(k^{-l-1+n}).$$
Moreover, for holomorphic vector field $Y$ and smooth function $\varphi$ on $M$ satisfying $d\varphi(X_{r})=d\varphi(Y_{r})=0,$ we have \begin{align*}F_{l}(X+k^{-1} Y,\omega+k^{-1} \ddbar \varphi)&=F_{l}(X,\omega)+k^{-1}(H_{Y}-\frac{1}{2}\inner{\nabla \varphi, \nabla H_{X}}_{\omega})\\&+O(k^{-2}).\end{align*} 
\end{thm}

\begin{proof}

By Prop. \ref{prop4} and \ref{prop5}, we have $$H_{k}(Q_{k}(f))\sim f+q_{1}(f)k^{-1}+ q_{2}(f)k^{-2}+\cdots,$$
$$H_{k}(A-c_{A}(k)I)\sim k(H_{X}+k^{-2}h_{2}+\cdots),$$
Therefore,  \begin{align*}&H_{k}(Q_{k}(F_{l}(X,\omega))-k^{-1}H_{k}(A-c_{A}(k)I)\\&=(f_{0}-H_{X})+(f_{1}+q_{1}f_{0})k^{-1}+\cdots(f_{l}-h_{l}+\sum_{i=1}^lq_{i}(f_{l-i}))k^{-l}+O(k^{-l-1})\\&=O(k^{-l-1}).\end{align*}
Since both $H_{k}(Q_{k}(f))$ and $H_{k}(Ac_{A}(k)I)$ have complete asymptotic expansion, there exists positive $c_{m}$ such that $$\abs{H_{k}(Q_{k}(\sum_{i=0}^{l}k^{-i}f_{i}))-k^{-1}H_{k}(A-c_{A}(k)I)}_{C^m} \leq c_{m}k^{-l-1}.$$ Applying Prop. \ref{prop6} to the finite dimensional subspace $$W:=\textrm{span}\{f_{i}(X)| X \in \Lie{g}, \, 1 \leq i \leq l\},$$ we have
$$tr \Big(\Big(Q_{k}(\sum_{i=0}^{l}k^{-i}f_{i}))-k^{-1}A -k^{-1}c_{A}(k)I\Big)^2\Big) \leq ck^{-l+1+n}.$$

For the second part, note that the asymptotic expansions in \ref{prop4} and \ref{prop5} are uniform with respect to $\omega.$ Therefore, the coefficients $f_{j}(X, \omega)$ are smooth with respect to $X$ and $\omega.$ This implies that 
$$f_{j}(X+k^{-1} Y,\omega+k^{-1} \ddbar \varphi)=f_{j}(X,\omega)+O(k^{-1}).$$
On the other hand, $f_{0}=f_{0}(X+k^{-1} Y,\omega+k^{-1} \ddbar \varphi)$ is the normalized Hamiltonian of $X+k^{-1} Y$ with respect to the K\"ahler metric $\omega+k^{-1} \ddbar \varphi.$ By definition, $$\bar{\partial} f_{0}=\iota_{X+k^{-1}Y} (\omega+k^{-1}\ddbar{\varphi}),$$ $$\int_{M}f_{0} (\omega+k^{-1}\ddbar{\varphi})^n=0.$$ Hence, Corollary \ref{cor,Hamiltonian} implies that
\begin{align*}\bar{\partial} f_{0}&=\iota_{X+k^{-1}Y} (\omega+k^{-1}\ddbar{\varphi})\\&=\iota_{X}(\omega+k^{-1}\ddbar{\varphi})+k^{-1}(\iota_{Y}(\omega))+O(k^{-2})\\&=\bar\partial{(H_{X}+k^{-1}(H_{Y}-\frac{1}{2}\inner{\nabla \varphi, \nabla H_{X}})+O(k^{-2}))}.\end{align*}
Now $\partial \bar{\partial}-$lemma and the normalization condition for $f_{0}$ conclude the proof.

\end{proof}

 \section{Constructing approximate solutions}

The goal of this section is to construct a sequence of "almost" relatively balanced embeddings. More precisely, for any positive integer $q$, we construct hermitian metrics $h(k)$ on $L$ and orthonormal bases $\underline{s}^{(k,q)}=(s_{0}^{(k,q)},...,s_{N}^{(k,q)})$ for $H^0(M,L^k)$ with respect to $\Hilb{(h(k))}$ such that $$\bar{\mu}(\underline{s}^{(k,q)}) \in V_{\underline{s}^{(k,q)}} \,\,\, \textrm{mod} \,\,\, k^{-q-1}.$$
Note that $\bar{\mu}(\underline{s}^{(k)})$ is given by 
$$\bar{\mu}(\underline{s}^{(k,q)})_{ij}=\int_{M} \inner{s_{i}^{(k,q)}, s_{j}^{(k,q)}}_{h_{FS}}\omega_{FS}^n,$$
where $h_{FS}=\FS{(\Hilb{(h(k))})}$ and $\omega_{FS}=\ddbar \log h_{FS}.$ By definition of $h_{FS}$, we have $h_{FS}=\rho_{k}(\omega(k))^{-1}h(k)^k$. Let $\omega(k)=\ddbar \log h(k).$ Therefore,

\begin{align*} \bar{\mu}_{k}(\underline{s}^{(k,q)})_{ij}&=\int_{M} \rho_{k}^{-1}\inner{s_{i}^{(k,q)}, s_{j}^{(k,q)}}_{h(k)^k}\big(k\omega(k)-\ddbar \log \rho_{k})\big)^n\\&= \int_{M} \frac{k^n}{\rho_{k}}\frac{(\omega(k)-k^{-1}\ddbar \log\rho_{k})^n}{\omega(k)^n}\inner{s_{i}^{(k,q)}, s_{j}^{(k,q)}}_{h(k)^k}\omega(k)^n,\end{align*} where $\rho_{k}=\rho_{k}(\omega(k))$ is the $k^{th}$ Bergman kernel of $\omega(k).$

\begin{Def}\label{def5}
For a K\"ahler metric $\omega$ on $M$, define $$\widetilde{\rho_{k}}(\omega)=  \frac{k^n}{\rho_{k}(\omega)} \frac{(\omega+k^{-1}\ddbar \log\rho_{k}(\omega))^n}{\omega^n}.$$ 

\end{Def}

An easy consequence Catlin-Tian-Yau-Zelditch expansion (\cite{T0}, \cite{C}, \cite{Z}, \cite{Lu}) is the following.

\begin{lem}\label{lem7}

For any K\"ahler metric $\omega$ on $M$ we have the following asymptotic expansion which holds in $C^{\infty}$,  $$\widetilde{\rho_{k}}(\omega) \sim 1+a_{1}(\omega)k^{-1}+a_{2}(\omega)k^{-2}+\cdots,$$ where $a_{1}(\omega)=-S(\omega)$. Here $S(\omega)$ is the scalar curvature of $\omega.$ Moreover, the expansion is uniform with respect to $\omega.$

\end{lem}
Note that if $\omega$ is $T$-invariant, then $\widetilde{\rho_{k}}(\omega)$ is $T$-invariant and therefore $a_{i}(\omega) \in (C^{\infty}(M))^T.$ Here $(C^{\infty}(M))^T$ is the space of $T$-invariant smooth functions on $M$. 
\begin{prop}\label{prop7}

Let $\omega_{\infty}$ be a $T$-invariant extremal metric in the class $2\pi c_{1}(L)$ and let $X_{\infty}$ be the corresponded extremal vector field, i.e. $\bar{\partial} S(\omega_{\infty})=\iota_{X_{\infty}}\omega_{\infty}.$ Let $l$ be a fixed large integer. Then there exists smooth functions $\varphi_{1}, \varphi_{2},\dots \in (C^{\infty}(M))^T$ and  holomorphic vector fields $X_{1}, X_{2}, \dots \in \Lie{t}^{\C} $ such that for any positive integer $q$, we have
$$\widetilde{\rho_{k}}(\omega(k))+k^{-1}F_{l}(X(k), \omega(k))=1+c_{1}k^{-1}+ \cdots c_{q}k^{-q}+O(k^{-q-1}),$$ where $c_{1},\dots c_{q}$ are constants, $F_{l}$ is defined in Definition \ref{defF} and $$\omega(k)=\omega_{\infty}+k^{-1}\ddbar \varphi_{1}+\cdots +k^{-q}\ddbar \varphi_{q},$$
$$X(k)=X_{\infty}+k^{-1}X_{1}+\dots +k^{-q}X_{q}.$$

\end{prop}

\begin{proof}


First note that the extremal vector field $X_{\infty} \in \Lie{t}^{\C},$ since $\omega_{\infty}$ is $T$-invariant and $\Lie{t}^{\C}$ is a maximal torus in $\Lie{g}$.
Applying Lemma \ref{lem7}, there exists $T$-invariant functions $a_{1}(\omega_{\infty}), a_{2}(\omega_{\infty}),\dots$ such that $$\widetilde{\rho_{k}}(\omega_{\infty})=1+a_{1}(\omega_{\infty})k^{-1}+a_{2}(\omega_{\infty})k^{-2}+\cdots .$$ Moreover, $a_{1}=-S(\omega_{\infty})$. By definition of extremal metrics, the gradient of $a_{1}$ is $-X_{\infty}.$ Note that the normalized Hamiltonian of $X_{\infty}$ with respect to $\omega_{\infty}$ is $$H_{\infty}:=H_{X_{\infty}}=S(\omega_{\infty})-\bar{s},$$ where $\bar{s}=\frac{1}{V}\int_{M} S(\omega_{\infty})\omega_{\infty}^n$ is the average of the scalar curvature. It is well known that the kernel of Lichnerowicz operator$\mathcal{D}^*\mathcal{D}$ consists of Hamiltonian functions on $M$ whose gradient is a holomorphic vector field. Therefore extremality of $\omega_{\infty}$ implies that $a_{1} \in Ker( \mathcal{D}^*\mathcal{D})$. Then we have, \begin{align*}\widetilde{\rho_{k}}(\omega_{\infty})&+k^{-1}F_{l}(X_{\infty},\omega_{\infty})\\&=1-S(\omega_{\infty})k^{-1}+O(k^{-2})+k^{-1}(H_{X_{\infty}}+O(k^{-1}))\\&=1-S(\omega_{\infty})k^{-1}+O(k^{-2})+k^{-1}((S(\omega_{\infty})-\bar{s})+O(k^{-1}))\\&=1+\bar{s}k^{-1}+\cdots.\end{align*}  Note that $\bar{s}$ is a topological constant only depends on the K\"ahler class $c_{1}(L).$
The linearization of scalar curvature at $\omega_{\infty}$ is given by 
$$L(\varphi)=\mathcal{D}^*\mathcal{D} \varphi -\frac{1}{2}\inner{\nabla \varphi, \nabla H_{\infty}}_{\omega_{\infty}}.$$ 
Thus,  \begin{align*}S(\omega_{\infty}&+k^{-1}\ddbar \varphi_{1})\\&=S(\omega_{\infty})+k^{-1}( \mathcal{D}^*\mathcal{D}\varphi_{1} -\frac{1}{2}\inner{\nabla \varphi_{1}, \nabla H_{\infty}} )+O(k^{-2}).\end{align*}
On the other hand, applying Theorem \ref{thm2} to holomorphic vector field $X_{1} \in \Lie{t}^{\C}$ and $\varphi_{1}\in (C^{\infty}(M))^T$, we have
 \begin{align*}F&_{l}(X_{\infty}+k^{-1}X_{1},\omega_{\infty}+k^{-1}\ddbar \varphi_{1})\\&=S(\omega_{\infty})-\bar{s}+k^{-1}(f_{1}(X_{\infty},\omega_{\infty})+H_{X_{1}}+\frac{1}{2}\inner{\nabla \varphi_{1}, \nabla H_{\infty}}) +O(k^{-2}).\end{align*}
Thus, 
 \begin{align*}&\widetilde{\rho_{k}}(\omega_{\infty}+k^{-1}\ddbar \varphi_{1})+k^{-1}F_{l}(X_{\infty}+k^{-1}X_{1},\omega_{\infty}+k^{-1}\ddbar \varphi_{1})\\&=1+\bar{s}k^{-1}+(a_{2}(\omega_{\infty})-\mathcal{D}^*\mathcal{D} \varphi_{1} +f_{1}(X_{\infty},\omega_{\infty})+H_{X_{1}}) k^{-2}+O(k^{-3}).\end{align*}
Since the image of  $\mathcal{D}^*\mathcal{D}$ is the orthogonal complement of Hamiltonians and $\mathcal{D}^*\mathcal{D}$ preserves $(C^{\infty}(M))^T$, we can choose $\varphi_{1} \in (C^{\infty}(M))^T$ and $X_{1} \in \Lie{t}^{\C}$ such that $$\mathcal{D}^*\mathcal{D} \varphi_{1} +H_{X_{1}}=a_{2}+f_{1} -\frac{1}{V}\int_{M} (a_{2}+f_{1}) \omega_{\infty}^n.$$ This implies that 
\begin{align*}\widetilde{\rho_{k}}(\omega_{\infty}+k^{-1}\ddbar \varphi_{1})+&k^{-1}F_{l}(X_{\infty}+X_{1},\omega_{\infty}+k^{-1}\ddbar \varphi_{1})\\&=1+\bar{s}k^{-1}+c_{2}k^{-2}+O(k^{-3}),\end{align*} for a constant $c_{2}$.
We can complete the proof using induction. Suppose we have chosen $\varphi_{1}, \dots\varphi_{q} \in (C^{\infty}(M))^T$ and  holomorphic vector fields $X_{1}, \dots X_{q} \in \Lie{t}^{\C}$ such that for 
$$\omega(k)=\omega_{\infty}+k^{-1}\ddbar \varphi_{1}+\cdots +k^{-q}\ddbar \varphi_{q},$$
$$X(k)=X_{\infty}+k^{-1}X_{1}+\dots +k^{-q}X_{q},$$ we have

$$\widetilde{\rho_{k}}(\omega(k))+k^{-1}F_{l}(X(k), \omega(k))=1+c_{1}k^{-1}+ \cdots c_{q}k^{-q}+b k^{-q-1}+O(k^{-q-2}),$$ where $c_{1},\dots ,c_{q}$ are constants and $b \in (C^{\infty}(M))^T.$ 
Using linearization of the scalar curvature at $\omega_{\infty}$, there exist polynomials $S_{1}, \dots ,S_{q+1}$  in $\varphi_{1}, \dots ,\varphi_{q}$ and their covariant derivatives such that  \begin{align*}S(\omega(k)+&k^{-q-1}\ddbar \varphi_{q+1})=S(\omega_{\infty})+k^{-1}S_{1}+k^{-2}S_{2}+\cdots +k^{-q}S_{q}\\&+k^{-q-1}(S_{q+1}+\mathcal{D}^*\mathcal{D}\varphi_{q+1} -\frac{1}{2}\inner{\nabla \varphi_{q+1}, \nabla H_{\infty}} )+O(k^{-q-2}).\end{align*}  
For example $$S_{1}=\mathcal{D}^*\mathcal{D}\varphi_{1} -\frac{1}{2}\inner{\nabla \varphi_{1}, \nabla H_{\infty}} ,$$ $$S_{2}=\mathcal{D}^*\mathcal{D}\varphi_{2} -\frac{1}{2}\inner{\nabla \varphi_{2}, \nabla H_{\infty}}  + \textrm{a quadratic term in } \varphi_{1}.$$ The same argument can be applied to all other coefficients of $\widetilde{\rho_{k}}.$ Thus,
\begin{align*}\widetilde{\rho_{k}}(\omega(k)+&k^{-q-1}\ddbar \varphi_{q+1})\\&=\widetilde{\rho_{k}}(\omega(k))+k^{-q-1}(-\mathcal{D}^*\mathcal{D}\varphi_{q+1} +\frac{1}{2}\inner{\nabla \varphi_{q+1}, \nabla H_{\infty}} )+O(k^{-q-2}).\end{align*}
On the other hand Theorem \ref{thm2} implies that
 \begin{align*}F_{l}&(X(k)+k^{-q-1}X_{q+1},\omega(k)+k^{-q-1}\ddbar \varphi_{q+1})\\&=F_{l}(X(k),\omega(k))+k^{-q-1}(H_{X_{q+1}}-\frac{1}{2}\inner{\nabla \varphi_{q+1}, \nabla H_{\infty}} ) +O(k^{-q-2}).\end{align*}
Hence, 
\begin{align*}\widetilde{\rho_{k}}&(\omega(k)+k^{-q-1}\ddbar \varphi_{q+1})\\&+k^{-1}F_{l}(X(k)+k^{-q-1}X_{q+1},\omega(k)+k^{-q-1}\ddbar \varphi_{q+1})\\&=\widetilde{\rho_{k}}(\omega(k))+k^{-1}F_{l}(X(k),\omega(k))\\&+k^{-q-1}(-\mathcal{D}^*\mathcal{D}\varphi_{q+1} +H_{X_{q+1}})+O(k^{-q-2})\\&=1+c_{1}k^{-1}+ \cdots c_{q}k^{-q}\\&+k^{-q-1}(-\mathcal{D}^*\mathcal{D}\varphi_{q+1} +H_{X_{q+1}}+b)+O(k^{-q-2}).\end{align*} 
Now, we can choose $\varphi_{q+1} \in (C^{\infty}(M))^T$ and $X_{q+1} \in \Lie{t}^{\C}$ such that $$\mathcal{D}^*\mathcal{D}\varphi_{q+1} -H_{X_{q+1}}=b-\frac{1}{V}\int_{M} b\omega_{\infty}^n.$$
\end{proof}



\begin{cor}\label{cor51}
Let $l$ be fixed large integer and $q \leq \frac{l-n-1}{2}$ be any positive integer. Let $\underline{s}^{(k,q)}=(s_{0}^{(k,q)}, \dots,s_{N_{k}}^{(k,q)}) $ be orthonormal ordered bases with respect to $L^2(h(k)^k, \omega(k)^n)$, where $\omega(k)$ is given in Proposition \ref{prop7} and $h(k)$ is a corresponding hermitian metric on $L$. Then there exist $B(k) \in V_{\underline{s}^{(k,q)}}(T)$ and constants $c(k)$ such that $$\norm{\bar{\mu}(\underline{s}^{(k,q)})-B(k)-c(k)I}_{op}=O(k^{-q-1}) ,$$ for $k \gg 0.$

\end{cor}

\begin{proof}
For simplicity we drop all superscripts $(k,q)$ in the proof. By Definition \ref{def5}, we have 
$$\bar{\mu}(\underline{s}^{(k,q)})_{ij}=\int_{M}\inner{s_{i}, s_{j}}_{h_{FS}}\omega_{FS}^n=\int_{M} \tilde{\rho}_{k}(\omega(k))\inner{s_{i}, s_{j}}_{h(k)^k}\omega(k)^n.$$ On the other hand, Proposition \ref{prop7} implies that there exist smooth function $\epsilon_{k}=O(1)$ and constants $c_{1}, \dots ,c_{q}$ such that $$\widetilde{\rho_{k}}(\omega(k))+k^{-1}F_{l}(X(k), \omega(k))=1+c_{1}k^{-1}+ \cdots +c_{q}k^{-q}+\epsilon_{k}k^{-q-1}.$$ Let $c_{1}(k)= 1+c_{1}k^{-1}+\cdots +c_{q}k^{-q}$. Hence, 
\begin{align*}  \bar{\mu}(\underline{s}^{(k,q)})_{ij}&= -k^{-1}\int_{M} F_{l}(X(k),\omega(k))\inner{s_{i}, s_{j}}_{h(k)^k}\omega(k)^n\\& +c_{1}(k)\int_{M} \inner{s_{i}, s_{j}}_{h(k)^k}\omega(k)^n+k^{-q-1}
\int_{M} \epsilon_{k}\inner{s_{i}, s_{j}}_{h(k)^k}\omega(k)^n\\&=-k^{-1}
(Q_{k})_{ij}+c_{1}(k)\delta_{ij}+k^{-q-1}\int_{M} \epsilon_{k}\inner{s_{i}, s_{j}}_{h(k)^k}\omega(k)^n,\end{align*} 
where $Q_{k}=Q_{k}(F_{l}(X(k), \omega(k))).$ Define the matrix $E$ by $$E_{ij}=\int_{M} \epsilon_{k}\inner{s_{i}, s_{j}}_{h(k)^k}\omega(k)^n.$$ Then

$$\bar{\mu}(\underline{s}^{(k,q)})+k^{-1}Q_{k}-c_{1}(k)I=k^{-q-1}E.$$

Let $A(k) \in V_{\underline{s}^{(k,q)}}(T) \subset \sqrt{-1}\Lie{su}(N_{k}+1)$ be the associated matrices to holomorphic vector fields $X(k)$ and let $B(k)=-k^{-1}A(k).$ Therefore, Theorem \ref{thm2} implies that there exists a constant $c_{2}(k)$ such that $$\norm{Q_{k}-A(k)-c_{2}(k)I}^2:=tr\Big(\Big(Q_{k}-A(k)-c_{2}(k)I\Big)^2\Big)=O(k^{-l-1+n}).$$ 
Let $c(k)=c_{1}(k)+c_{2}(k).$ Thus, \begin{align*}\norm{\bar{\mu}(\underline{s}^{(k,q)})-B(k)-c(k)I}_{op}&\leq\norm{\bar{\mu}(\underline{s})+k^{-1}Q_{k}-c_{1}(k)I}_{op}\\&+ \norm{ k^{-1}Q_{k}+B(k)-c_{2}(k)I}_{op}\\&
\leq  k^{-q-1}\norm{E}_{op}+k^{-1}\norm{Q_{k}-A(k)-c_{2}(k)I}_{op}\\& \leq k^{-q-1}\norm{E}_{op}+k^{-1}tr\big((Q_{k}-A(k)-c_{2}(k)I)^2\big)^{\frac{1}{2}}\\ &\leq k^{-q-1}\norm{E}_{op}+O(k^{\frac{-l-3+n}{2}})\\& \leq ck^{-q-1}\norm{E}_{op},\end{align*}
since $q \leq \frac{l-n-1}{2}$. On the other hand, an argument of Donaldson (\cite[Prop. 27]{D1}, \cite [Lemma 15]{F0}) implies that $$\norm{E}_{op} \leq \norm{\epsilon_{k}}_{C^{0}}=O(1).$$ 
This concludes the proof.

\end{proof}

\section{proof of the Theorem \ref{mainthm}}

In order to prove the main theorem, we follow \cite{Ma3} and \cite{PS3}. 

As before let $G=\widetilde{Aut}_{0}(M)$ be the group of Hamiltonian automorphisms of $M$. Let $T$ be a maximal compact torus in $G$ and $T^{\C}$ be its complexification in $G$. Suppose that $\omega_{\infty}$ is a $T$-invariant extremal metric on $M$ in the class of $2\pi c_{1}(L).$ The following two lemmas are straightforward from the formalism of relative stability developed in \cite{Sz1}. 

\begin{lem}\label{lem9}
For any $\underline{s} \in \mathcal{B}_{k}^T$, we have $$\bar{\mu}_{k}(\underline{s}) \in \Lie{s}_{T}.$$ Here $\Lie{s}_{T}= \sqrt{-1} \Lie{su}(N_{k}+1)\bigcap \Lie{s}_{T}^{\C}$ (c.f. \eqref{eqST}). 

\end{lem}

\begin{lem}\label{lem10}
Let $\sigma \in S_{T}$ and $\underline{s}^{\sigma}=\sigma .\underline{s}(k)$, i.e.
$s^{\sigma}_{i}=\sum_{j=0}^{N_{k}}\sigma_{ij}s_{j}.$
Then  $$V_{\underline{s}^{\sigma}}(T)=V_{\underline{s}(k)}(T).$$ 
\end{lem}

For large positive integers $l$ and  $q,$ Proposition \ref{prop7} implies that there exists $\varphi_{1}, \dots ,\varphi_{q} \in (C^{\infty}(M))^T$ and holomorphic vector fields $X_{1}, \dots ,X_{q} \in \Lie{t}^{\C}$ such that 
$$\widetilde{\rho_{k}}(\omega(k))+k^{-1}F_{l}(X(k), \omega(k))=\textrm{constant}+O(k^{-q-1}).$$ Here $\omega(k)=\omega_{\infty}+\sum_{i=1}^{q}k^{-i}\ddbar \varphi_{i}$ and $X(k)=X_{\infty}+\sum_{i=1}^{q}k^{-i}X_{i}.$ Let $h(k)$ be a  hermitian metric on $L$ such that $\ddbar \log h(k)=\omega(k).$

Let $\underline{s}^{(k,q)}=(s_{0}^{(k,q)}, \dots ,s_{N_{k}}^{(k,q)})$ be a sequence of $L^2(h(k)^k, \omega(k))$-orthonormal ordered bases. 
Therefore, Corollary \ref{cor51} implies that $$\norm{\bar{\mu}(\underline{s}^{(k,q)})-B(k)-c(k)I}_{op}=O(k^{-q-1}) \, \textrm{for} \,\, k \gg 0,$$ where $B(k) \in V_{\underline{s}^{(k,q)}}(T)$ and $c(k)$ is constant. 

Let $r \geq 4$ be an integer. From now on, we fix integers $l$ and $ q$ satisfying $\frac{n+4+r}{2} \leq q \leq \frac{l-n-1}{2}.$ We also fix $L^2(h(k)^k, \omega(k))-$orthonormal ordered bases $\underline{s}^{(k,q)}.$ To simplify the notation, we let  $\underline{s}(k)=\underline{s}^{(k,q)}.$ Since, the $L^2$ norms on $H^{0}(M,L^k)$ are $T$-invariant, we may assume that 
$\underline{s}(k)=(s_{0}^{(k,q)}, \dots ,s_{N_{k}}^{(k,q)})$ are compatible with the splitting \eqref{eq61} and provide orthogonal bases on each $E(\chi_{i}).$

The proof of the following can be found in \cite{D3}.
\begin{lem}(\cite[Prop. 27]{D1}, \cite[Lemma 3]{PS3})\label{lem8}
Let $r \geq 4$ be an integer. There exists $C=C_{r} > 0$ with the following properties. Let $A \in \sqrt{-1}\Lie{su}(N_{k} + 1)$, with
$\norm{A}_{op}\leq1$  and, for $\abs{t}\leq\frac{1}{10}$  with $k \gg0$, let $\sigma_{t} = e^{tA}$ and $\widetilde{\omega}_{\infty}=k\omega_{\infty}$. Then $$\norm{\iota_{\underline{s}(k)}^*\sigma_{t}^*\omega_{FS,\PNK}-\widetilde{\omega}_{\infty}}_{C^r(\widetilde{\omega}_{\infty})} \leq Ct+O(k^{-1}).$$ \\ 
  Moreover, $$\norm{k^{-1}\iota_{\underline{s}(k)}^*\sigma_{t}^*\omega_{FS,\PNK}-\omega_{\infty}}_{C^r(\omega_{\infty})} \leq Ck^{\frac{r+2}{2}}t+O(k^{-1}).$$
\end{lem}

\begin{proof}[Proof of Theorem \ref{mainthm}]

In the following proof, we fix a large $k$ and the embedding $\iota_{\underline{s}(k)}:M \to \mathbb{P}^{N_{k}}$. 
Let $A \in V_{\underline{s}(k)}(T)^{\bot}$ such that $\norm{A}=1$ and let $\sigma_{t}=e^{tA}.$ By definition $A \in \sqrt{-1} \Lie{su}(N_{k}+1)\bigcap \Lie{s}_{T}^{\C}$ and is prependicular to $V_{\underline{s}(k)}(T).$ Therefore, $tr(A)=tr(AB(k))=0.$ Define $f_{A}(t)=\mathcal{F}(tA)$ (c.f. \eqref{eq,F}). 
Thus Lemma \ref{lem3} and Corollary \ref{cor51} imply that \begin{align}  \abs{\dot{f}_{A}(0)}=&\abs{\int_{M}tr(A\mu_{k} )\omega_{FS}^n}=\abs{tr(A\bar{\mu}_{k})}=\abs{tr\Big(A\big(\bar{\mu}_{k} -B(k)-c(k)I\big)\Big)}\nonumber\\ \label{eq62}& \leq \sqrt{N_{k}+1} \norm{A}\norm{\bar{\mu}_{k} -B(k)-c(k)I}_{op}\\& \leq Ck^{\frac{n}{2}-q-1}.\nonumber\end{align} Here, $\mu_{k}=\mu_{k}(\underline{s}(k))=\mu(\underline{s}^{(k,q)})$ and $\bar{\mu}_{k}=\int_{M}\mu(\underline{s}^{(k,q)}) \omega_{FS}^n.$\\

On the other hand, Lemma \ref{lem8} implies that there exists $\delta >0 $ such that $\iota_{\underline{s}(k)}^*\sigma_{t}^*\omega_{FS}$ has $2$-bounded geometry for $\abs{t} \leq \delta$ and $k \gg 0.$ Therefore, Lemma \ref{lem3}, Lemma \ref{lem10} and Theorem \ref{thm3} imply that
there exists for $\abs{t} \leq \delta$, we have
\begin{align}\ddot{f}_{A}(t)&=\int_{M} \norm{\pi_{\mathcal{N}} \xi_{A}}^2_{FS}\iota_{\underline{s}(k)}^*\sigma_{t}^*\omega_{FS}^n =\norm{\pi_{\mathcal{N}} \xi_{A}}^2_{L^2(\iota_{\underline{s}(k)}^*\sigma_{t}^*\omega_{FS})}\nonumber\\ \label{eq63}&\geq ck^{-2}\norm{A}^2=ck^{-2}.\end{align}
Let $\delta_{k}=k^{-\frac{r+4}{2}}\delta.$ Since $q \geq \frac{n+r+4}{2}+2$, then \eqref{eq62} and \eqref{eq63} imply that for $k \gg 0$, the function $f_{A}(t)$ is decreasing on $(-\infty, -\delta_{k})$ and increasing on $(\delta_{k}, \infty)$. Note that $\ddot{f}_{A}(t) \geq 0$ for all $t \in \R.$ Therefore, $$f_{A}(t) > f_{A}(0), \,\,\,\, \textrm{for any } t \,\, \textrm{such that }\,\, \abs{t}\geq \delta_{k}.$$ Note that \begin{align*}f_{A}(\delta_{k})-f_{A}(0)&=\int_{0}^{\delta_{k}}\int_{0}^{t}\ddot{f}_{A}(s)ds dt+\dot{f}_{A}(0) \delta_{k} \\&\geq ck^{-2}\frac{\delta_{k}^2}{2}-Ck^{-3-\frac{r+4}{2}}\delta_{k}>0.\end{align*} This implies that $f_{A}(t)$ archives its absolute minimum on $[-\delta_{k},\delta_{k}].$
Hence, for any $A \in V_{\underline{s}(k)}(T)^{\bot}$ with $\norm{A}=1$, there exists $t_{A} \in [-\delta_{k},\delta_{k}]$
such that $f_{A}(t) \geq f_{A}(t_{A})$ for all $t \in \mathbb{R}.$ Therefore, we have 
$$\mathcal{F}(A) \geq \inf \{f_{B}(t)| B \in V_{\underline{s}(k)}(T)^{\bot}, \norm{B}=1, \abs{t}\leq \delta_{k}\} \geq -C.$$
 Thus restriction of $\mathcal{F}$ to $ V_{\underline{s}(k)}(T)^{\bot}$ has a minimum at $B \in V_{\underline{s}(k)}(T)^{\bot}$. Note that $$\norm{B}_{op} \leq \norm{B} \leq \delta_{k}.$$
Let $\sigma=e^{B}$. Then, for any $A \in V_{\underline{s}(k)}(T)^{\bot}$, 
 we have $$tr(A \int_{M}\sigma^*(\mu \omega^n))=0.$$ Lemma \ref{lem8} and \ref{lem9} imply that $\int_{M}\sigma^*(\mu \omega^n) \in \Lie{s}_{T}.$ Therefore,  $$\int_{M}\sigma^*(\mu \omega^n) \in  V_{\underline{s}(k)}(T).$$ Therefore, 
$\widetilde{\omega}_{k}:=\iota_{\underline{s}(k)}^*\sigma^*\omega_{FS,\PNK}$ is relatively balanced. Moreover, Lemma \ref{lem8} implies that $$\norm{k^{-1}\widetilde{\omega}_{k}-\omega_{\infty}}_{C^r(\omega_{\infty})} \leq Ck^{\frac{r+2}{2}}\delta_{k}+O(k^{-1})=O(k^{-1}).$$
This shows that the sequence of rescaled relative balanced metrics $\omega_{k}:=k^{-1}\widetilde{\omega}_{k}$ converges to $\omega_{\infty}$ in $C^r$-norm. 
\end{proof}

\section{Applications}
\subsection{Uniqueness of extremal metrics} By a conjecture of X. X. Chen extremal metrics in any K\"ahler class, if exist, are unique up to automorphisms (c.f. \cite{CT}). The conjecture was proved by Berman-Berndtsson (\cite{BB}). Using approximation with relative balanced metrics, one can give another proof  for the uniqueness in any polarization.

\begin{thm} (Berman-Berndtsson, \cite{BB})
Let $(M,L)$ be a polarized manifold. Let  $\omega_{\infty}$, $\omega_{\infty}^{\prime}$ be extremal K\"ahler metrics in the class of
$2\pi c_{1}(L)$. Then there exists an automorphism $\Phi \in \widetilde{Aut}_{0}(M)$ such that $\omega_{\infty}^{\prime}=\Phi^*\omega_{\infty}.$

\end{thm}

\begin{proof}

Theorem \ref{mainthm} implies that there exists sequences of relatively balanced metrics $\omega_{k}$ and $\omega_{k}^{\prime}$ such that $\omega_{k} \to \omega_{\infty}$ and $\omega_{k}^{\prime} \to \omega_{\infty}^{\prime}$ in $C^{\infty}$. The uniquness of relatively balanced metrics implies that there exists a sequence of automorphisms $\Phi_{k} \in \widetilde{Aut}_{0}(M)$ such that $\omega_{k}^{\prime}=\Phi_{k}^*\omega_{k}$ (Proposition \ref{propuniquness}). An easy linear algebra argument implies that $\Phi_{k}$ has a convergent subsequence since $\omega_{k} $ and $\omega_{k}^{\prime}$ are convergent. This concludes the proof

\end{proof}

\subsection{Asymptotically Chow poly-stability of cscK polarizations }

In \cite{Ma1}, Mabuchi introduced an obstruction to asymptotic Chow semi-stability. Then he proved that if such obstructions vanish, then a polarization that admits a cscK metric is asymptotically Chow poly-stable. Mabuchi's result can be proved as an application of Theorem \ref{mainthm}. The obstruction to asymptotically semi stability introduced by Mabuchi is technical and is related to the isotropy action for $(M,L)$ (\cite[p.p. 463-464]{Ma1}). However, one can see that it is equivalent to the fact that the center of $G=\widetilde{Aut}_{0}(M)$ acts on the Chow line of $(M,L^k)$ trivially for $k \gg 0$. On the other hand, a theorem of Futaki (\cite[Proposition 4.1, Theorem 1.2]{Fu}) implies that it is equivalent to triviality of the action of the group $G$ on the Chow line of $(M,L^k)$.
First we prove the following Lemma.

\begin{lem}\label{lem,application1}
Suppose $M \subset \mathbb{P}^N$ is relatively balanced. If it is not balanced, then its Chow point is unstable.
\end{lem}

\begin{proof}

Let $$A:=\int_{ M} \frac{z_{i}\bar{z}_{j}	}{\abs{z}^2}\omega_{FS}^n. $$ It induces a one parameter subgroup $\sigma_{t}=e^{-tA}$ in $SL(N+1).$ Since $M \subset \mathbb{P}^N$ is relatively balanced, $\sigma_{t}(M)$ and therefore, $\sigma_{t}$ gives a one parameter subgroup of automorphisms of $M$. Define $$f(t)=\mathcal{F}(tA).$$ A theorem of Zhang(\cite{Z}) implies that $$f(t)=\log \frac{\norm{\sigma_{t} f_{M}}}{\norm{f_{M}}},$$ where $f_{M} \in H^{0}(Gr(N-n-1,  \mathbb{P}^N),\mathcal{O}(d))$ is the Chow point of $M \subset \mathbb{P}^N$ and $\norm{}$ is a norm defined on  $H^{0}(Gr(N-n-1,  \mathbb{P}^N),\mathcal{O}(d))$ (c.f. \cite{PS1}). Kempf-Ness implies that  $M$ is Chow semi-stable only if $f(t)$ is bounded from below. 

On the other hand, the change of variable formula for integrals implies that $$f^{\prime}(t)=Tr \Big(A \int_{ \sigma_{t}(M)} \frac{z_{i}\bar{z}_{j}	}{\abs{z}^2}\omega_{FS}^n\Big)=Tr \Big(A \int_{ M} \frac{z_{i}\bar{z}_{j}	}{\abs{z}^2}\omega_{FS}^n\Big)=Tr(A^2).$$ Therefore, $f(t)= Tr(A^2)t+c$ which is not bounded from below if $A \neq 0$. Therefore, the Chow point of $M$ is strictly unstable.

\end{proof}

\begin{thm}(Mabuchi, \cite[Main Theorem]{Ma3})
Let $(M,L)$ be a polarized manifold. Assume that $M$ admits a constant scalar curvature K\"ahler metric in the class of $2\pi c_{1}(L).$ If the group $G=\widetilde{Aut}_{0}(M)$ acts on the Chow line of $(M,L^k)$ trivially for $k \gg 0$, then $(M,L^k)$ is Chow poly-stable for $k \gg 0$
 
\end{thm}

\begin{proof}

By Theorem \ref{mainthm}, there exists a sequence of relatively balanced metrics on $(M,L^k)$ for $k \gg 0$. The group $G=\widetilde{Aut}_{0}(M)$ acts on the Chow line of $(M,L^k)$ trivially for $k \gg 0$ and therefore dose not destabilizes the Chow point of $(M,L^k)$ for $k \gg 0$. Hence Lemma \ref{lem,application1} implies that these relatively balanced metrics are indeed balanced for $k \gg0$. This implies that $(M,L)$ is asymptotically Chow poly-stable. 
\end{proof}

\subsection{Extremal metrics on products}

By a result of Yau (\cite{Y4}), any K\"ahler-Einstein metric on the product of compact complex manifolds is a product of K\"ahler-Einstein metric on each factor. It was generalized to extremal K\"ahler metrics on product of polarized compact complex manifolds by  Apostolov and Huang (\cite[Theorem 1]{AH}). They used the notion of relatively balanced metrics introduced by Mabuchi and proved the splitting result under some mild conditions on automorphism group. They also prove that one can drop the condition in presence of Theorem \ref{mainthm}.   

\begin{thm}(Apostolov-Huang, \cite{AH})
Let $M_{1}, \cdots M_{r} $ be compact projective manifolds polarized by ample holomorphic line  bundles $L_{1}, \cdots L_{r}$ respectively. Suppose $\omega_{\infty}$ is an extremal K\"ahler metric on $M$ in the class of $2\pi c_{1}(L)$, where $M=M_{1} \times \cdots \times M_{r}$ and $L =L_{1} \otimes \cdots \otimes L_{r}$. Then there exist extremal metrics $\omega_{\infty, i}$ on $M_{i}$ in the class of $2\pi c_{1}(L_{i})$ such that $\omega_{\infty}$ is the Riemanian product of $\omega_{\infty, 1}, \cdots \omega_{\infty,r}$.
\end{thm}


\end{document}